\RequirePackage{fix-cm}
\documentclass[twocolumn]{svjour3}          
\smartqed  
 
\usepackage[intlimits,sumlimits]{amsmath}
\usepackage{mathtools}
\usepackage{amssymb,amsfonts,amsbsy,bm}
\usepackage{colonequals}
\usepackage{graphicx}
\usepackage{placeins}
\usepackage{exscale}
\usepackage[hang]{subfigure}
\usepackage{booktabs,colortbl}
\usepackage{natbib}
\usepackage{mathrsfs}
\usepackage{afterpage}
\usepackage{makecell}

\usepackage[textwidth=40mm]{todonotes}

\usepackage{tikz}
\usetikzlibrary{patterns,positioning,shapes.geometric}
\usepackage{pgfplots}
\pgfplotsset{compat=1.12} 
\usepackage{pgfplotstable}

\usepackage[plain,vlined,longend,linesnumbered,commentsnumbered,ruled]{algorithm2e}

\newcommand{\thref}[1]{Theorem~\ref{#1}}
\newcommand{\secref}[1]{Section~\ref{#1}}
\newcommand{\figref}[1]{Figure~\ref{#1}}
\newcommand{\tabref}[1]{Table~\ref{#1}}

\newcommand{\bb}[1]{\boldsymbol{#1}}

\newcommand{\RR}{\mathbb{R}}
\newcommand{\dx}{\,\text{dx}}

\DeclareMathAlphabet{\mathpzc}{OT1}{pzc}{m}{it} 

\makeatletter
\let\@latex@error\@gobble
\makeatother

\begin{document}

\title{Topology optimization with worst-case handling of material uncertainties}


\author{Jannis Greifenstein \and 
        Michael Stingl       
}


\institute{Jannis Greifenstein \and Michael Stingl \at 
              Friedrich-Alexander-Universit\"at Erlangen-N\"urnberg (FAU),\\
              Mathematical Optimization, Department of Mathematics,\\
              Cauerstr. 11,
              91058 Erlangen, Germany. \\
              \email{jannis.greifenstein@fau.de, michael.stingl@fau.de}          
              %
}

\date{Received: date / Accepted: date}

\sloppy
\maketitle

\begin{abstract}
In this article a topology optimization method is developed, which is aware of material uncertainties. The uncertainties are handled in a worst-case sense, i.e. the worst possible material distribution over a given uncertainty set is taken into account for each topology. The worst-case approach leads to a minimax problem, which is analyzed throughout the paper. A conservative convex relaxation for the inner maximization problem is suggested, which allows to treat the minimax problem by minimization of an optimal value function. A Tikhonov type and a barrier regularization scheme are developed, which render the resulting minimization problem continuously differentiable. The barrier regularization scheme turns out to be more suitable for the practical solution of the problem, as it can be closely linked to a highly efficient interior point approach used for the evaluation of the optimal value function and its gradient. Based on this, the outer minimization problem can be approached by a gradient based optimization solver like the method of moving asymptotes. Examples from additive manufacturing as well as material degradation are examined, demonstrating the practical applicability as well as the efficiency of the suggested method. Finally, the impact of the convex relaxation of the inner problem is investigated. It is observed that for large material deviations, the robust solution may become too conservative. As a remedy, a RAMP-type continuation scheme for the inner problem is suggested and numerically tested.

\end{abstract}

\section{Introduction}
In the last decade, structural optimization has seen a vast increase in
publications considering uncertainties in the optimization problem.
Uncertainties in structural optimization can arise from a multitude of different
sources, such as modelling errors, inaccurate loading scenarios, manufacturing
precision or defective material. In this work, we will focus on the latter
case for compliance problems in linear elasticity.\\
Local material failure can be caused by unexpectedly high loads, leading to
yielding and a local loss in stiffness. Other sources of local material
degradation may be corrosion, fatigue failure or accidents. Furthermore, the
fabrication process may lead to variances in the achieved material
properties. In the linear elasticity setting, all of these material defects may
be modeled by either a reduced material stiffness or a removal of material
accounting for a complete failure of material. \\
In the literature, a number of publications can be found on topology 
optimization under consideration of uncertainties in geometry or material 
properties. Topological uncertainties are considered in \cite{Wang:11:Robust} for 
an etching fabrication process through uniform over- and underetching. 
In \cite{chen2010level}, uncertainties in the loads or in the material
parameters are represented with random fields using a Karhunen-Loeve expansion for a 
level-set method. In \cite{lazarov2012topology}, the material properties and the
topology are subject to random field uncertainties using a collocation method for 
density-based topology optimization. Both of the latter methods are able to yield robustly
optimized structures, but suffer from a seriously increasing computational
complexity if the distribution of the uncertain degradation is allowed to vary
not only on a broad scale over the body. Furthermore, even if the uncertainty
distribution is sufficiently well known and can be represented properly using the
random field uncertainties, the methods may only give a probability for the
actual structural performance. In this work, on the other hand, we would like to
obtain an actual bound for the structural performance, even if the allowed
uncertainties in the material properties are distributed in the worst possible
way. One method to obtain worst-case structures accounting for total material failure
is given in \cite{jansen2014topology}. Here, the design domain is subdivided
into finite element (FE) patches and, in each iteration, the material is removed for the FE patch in
a worst-case manner. However, one FE simulation is needed per considered FE
patch, so a fine resolution of material failure or allowing for the failure of
a combination of multiple of the considered FE patches would lead to a
prohibitively high computational complexity. In \cite{bendsoe1998method}, a topology design problem in a
continuum with worst-case material defects is solved using a relaxation to a
two-scale problem with material mixtures of damaged and undamaged material.
However, this is not done without the authors stating that a mathematically
consistent model was not pursued. More publications studying worst-case optimization models can
be found for truss optimization problems. In \cite{calafiore2008optimization},
robust truss optimization problems for uncertainties of polytopic shape in load and Young's
modulus are reformulated as convex semi-definite programs. The complexity of the
reformulated problem, however, increases drastically with the number of vertices
of the uncertainty polytope. The method thus only allows for the solution of
problems with relatively low-dimensional uncertainty characteristics.
In
\cite{achtziger1999optimal}, a bilevel truss
optimization is discussed for minimum compliance with worst case material degradation. In fact, it is shown that the
robust truss topology stays the same for the used material degradation model and
pure cross-sectional bar area optimization without upper bounds. \\
Contrary to the discussed publications, the aim of this work is to provide a
mathematically rigorous model which is capable of solving continuum topology
optimization problems with a worst-case approach for uncertain material properties, covering local failure as well as manufacturing uncertainties arising, e.g., from powderbed based additive manufacturing technologies.

The structure of the paper is as follows: In section 2 we develop a minimax problem, which is well defined in the sense that the inner maximization problem can be solved to global optimality. In section 3 differentiability properties for the minimax problem are discussed. Moreover two regularization approaches are suggested. In both cases the marginal function associated with the minimax problem can be shown to be continuously differentiable. In section 4 a complete algorithm for the solution of the regularized minimax problem is established. Finally, in section 5 numerical examples are presented for which uncertainties arise from additive manufacturing as well as local damage.

\section{Problem description}

\subsection{State problem}
The worst-case problem is discussed solely on the discretized level, i.e. we assume that the design domain $\Omega \subset \mathbb{R}^2$ is discretized in $n$ elements $\Omega_i$, on which material properties and densities are constant. Furthermore we introduce a discrete convolution operator as follows: for an element $e$ we define the set $V(e)$ of neighbour elements whose
distance is less than a given radius $R>0$, and $c_e$ the coordinates of the center of element
$e$. Then, for a discretized material density distribution $\rho \in \mathbb{R}^n$, the filtered value on element $e$ becomes \small{
\begin{equation}\label{eq:discrete-filter}
  (F*\bb{\rho})_e = \frac{\sum_{j\in V(e)}\rho_j\int_j
  F(\bb{x}-\bb{c}_e)\dx}{\sum_{j\in V(e)}\int_j
  F(\bb{x}-\bb{c}_e)\dx}
  =:\frac{\sum_{j\in
  V(e)}\rho_jF_{ej}}{\sum_{j\in V(e)}F_{ej}}.
\end{equation}}
We abbreviate the filtered
pseudo-density variable by $\bb{\tilde{\rho}}:=F*\bb{{\rho}}$. The state
problem then takes the form
\begin{equation}\label{eq:discrete-state}
  \bb{K}(\bb{\rho},\bb{\delta})\bb{u}=\bb{f}
\end{equation}
where $\bb{f}$ is the numerically integrated load form, $\bb{u}$ the
displacement vector and $\bb{K}$ is
the global stiffness matrix assembled by the local stiffness matrices
$\bb{K}=\sum_{e=1}^n\bb{K}_{e}$. The local stiffness matrices are
\begin{equation}\label{eq:discrete-stiff-assemble}
 \bb{K}_e(\bb{\rho},\bb{\delta})=\sum_{l=1}^{n_{ig}}\bb{B}_{el}^\top
\bb{C}_e\left(\bb{\rho},\bb{\delta}\right)\bb{B}_{el},
\end{equation}
where $\bb{B}_{il}$, $l=1,\ldots,n_{ig}$ denote standard
strain displacement matrices and $n_{ig}$ the number of integration points. Finally, the precise form of the material property function $\bb{C}_e\left(\bb{\rho},\bb{\delta}\right)$, which depends on the  pseudo-density $\bb{{\rho}} \in \mathbb{R}^n$ as well as on an uncertainty vector $\bb{{\delta}} \in \mathbb{R}^n$ is discussed in \secref{sec:material-model}.
\\
\subsection{Material model}\label{sec:material-model}
In additive manufacturing, not only the averaged material properties are
anisotropic. The variations in the
material properties discussed in this section can exhibit a different behaviour
in different directions as well. Commonly, the variation of the elastic modulus in
the build direction of a part is expected to be slightly larger than the
variation within the build layer (see e.\,g.
\cite{vdihandbuch,wegner2012betrachtung}). Nevertheless, we choose to use
an isotropic model to keep our final model numerically manageable. The chosen
formulation is based on a publication about truss topology design with
degradation effects (\cite{achtziger1999optimal}), in which the degradation is considered as a decrease of material stiffness. For two elasticity tensors $\bb{C^0}, \bb{C^1}$ the effective material property is computed by interpolation as
\begin{equation}\label{eq:damage-model}
\bb{\hat{C}}(\alpha)=(1-\alpha)\bb{C^0}+\alpha\bb{C^1}
\end{equation}
with $\alpha\in [0,1]$. This simplified model was used for an elastic continuum
by \cite{francfort1993stable} and can be traced back to \cite{kachanov1958time}
and \cite{lemaitre1978theorie}. Achtziger and Bends{\o}e first used this model
for damage in an optimization problem for maximum compliance in a truss
structure in \cite{achtziger1995design}. They revisited the formulation in
\cite{achtziger1998bounds} and were able to obtain lower and upper damage bounds by studying inner and outer approximations. This same inverse interpolation was used in \cite{achtziger1999optimal} for a basic truss topology design problem. Applied to a continuum setting, it
reads
\begin{equation}\label{eq:matdelta}
  \hat{\bb{C}}_e(\bb{\delta})=\left( \frac{1-{\delta}_e
}{E_0}+\frac{{\delta}_e}{E_D}\right)^{-1}\bb{C}
\end{equation}
where $e$ denotes the finite element number, $E_0>E_D>0$ the largest and
smallest Young's moduli allowed and
\begin{equation}\label{eq:unit-material-tensor}
\bb{C} = \frac{1}{(1+\nu)(1-2\nu)}  
  \begin{pmatrix}
    1-\nu & \nu &  \nu & 0 \\
    \nu & 1-\nu &  \nu & 0 \\
    0 & 0 & 0 & (1-2\nu)/2 \\
    \end{pmatrix}
\end{equation}
a unit stiffness tensor with the Poisson's ratio $\nu$ of the material. The
parameter ${\delta}_e$ is used to interpolate between the least stiff and the
stiffest material and is constant on each finite element $e$.
This parameter represents the uncertainty in the material properties. A major
advantage of this formulation will be the concavity of the compliance functional with respect to $\bb{\delta}$ in the optimization problem introduced in the next subsection. Furthermore, the material stiffness is strictly decreasing in ${\delta}_e$, so ${\delta}_e$ can be interpreted as a material degradation parameter.
\\
As usual in the classic SIMP approach, the topology itself is represented
through a pseudo-density variable $\bb{\rho}$. Consequently, the formula for the material property function becomes
\begin{equation}\label{eq:locmat}
  \bb{C}_e(\bb{\rho},\bb{\delta}) = \tilde{\rho}_e^p \hat{\bb{C}}_e(\bb{\delta})
\end{equation}
where $\hat{\bb{C}}_e$ is defined as in (\ref{eq:matdelta}), $\tilde{\rho}$ denotes the
pseudo-density after application of the convolution and $p>1$ is a penalty parameter.\\
\subsection{Optimization problem}\label{sec:optproblem}
We start out with the formulation of a minimax problem, where the compliance is minimized w.\,r.\,t. the topology $\bb{\rho}$ and maximized
w.\,r.\,t. the material uncertainties $\bb{\delta}$:
\begin{align}\label{eq:problem-state}
\left\{
\begin{array}{l}
  \min\limits_{\bb{\rho}\in U_{\text{ad}}}\max\limits_{\bb{\delta}\in
  \Delta_{\text{ad}}}\bb{f}^\mathrm{T}
  \bb{u}(\bb{\rho},\bb{\delta})
  \\
  [3mm]\,\mbox{s.t.:}\ 
  \bb{K}(\bb{\rho},\bb{\delta})\bb{u}=\bb{f}.\\
\end{array}
\right. 
\end{align}
Here, $U_{\text{ad}}$ and $\Delta_{\text{ad}}$ are the admissible design sets
containing bound and volume constraints, $\bb{K}(\bb{\rho},\bb{\delta})$
the symmetric and positive definite stiffness matrix, $\bb{u}$ the vector
containing the nodal displacements and $\bb{f}$ the external load vector.
The stiffness matrix $\bb{K}$ is assembled analogously to (\ref{eq:discrete-state}) with $C_e$ as in
(\ref{eq:locmat}), if not stated differently. 
The admissible set of pseudo-densities is given as  
\begin{equation}\label{eq:U-ad}
\begin{aligned}
  U_{\text{ad}} = \bigg\{ \bb{\rho}\in\RR^n :&\ \rho_{\text{min}}\leq
  \rho_e \leq 1,\ e=1,\ldots,n,\\ 
  &\sum_{e=1}^n v_e\rho_e/|\Omega| \leq V \bigg\},
\end{aligned}
\end{equation}
where $\rho_{\text{min}}$ is a small positive number, $v_e$ is the volume of
element $\Omega_e$ and $V$ is an upper bound on the admissible volume. Further, we use the uncertainty sets
\begin{equation}\label{eq:delta-ad}
\begin{aligned}
  \Delta_{\text{ad}} = \bigg\{ \bb{\delta}\in\RR^n :&\ 0\leq \delta_e \leq 1,\ e=1,\ldots,n,\\
  &\sum_{e=1}^n v_e\delta_e/|\Omega| = D \bigg\},
\end{aligned}
\end{equation}
or alternatively
\begin{equation}\label{eq:delta-ad-rho}
\begin{aligned}
\Delta_{\text{ad}}(\bb{\rho}) = \bigg\{ \bb{\delta}\in\RR^n :&\ 0\leq \delta_e \leq 1,\ e=1,\ldots,n,\\
&\sum_{e=1}^n v_e \tilde{\rho}_e^p\delta_e/|\Omega| = D \bigg\},
\end{aligned}
\end{equation}
with $D$ being an upper bound on the total material uncertainty. Note that for a given binary
distribution $\bb{\tilde{\rho}}$ the two definitions of $\Delta_{\text{ad}}$ will yield similar solutions of the inner maximization problem, as material deviations will naturally be concentrated on elements where $\tilde{\rho}_e=1$.
The first definition avoids the dependence on the upper level optimization variable $\bb{\rho}$, while the second definition is more rigorous from an application point of view.
\begin{remark}
In the application context of additive manufacturing (see section 5.2), we will further study the impact of additional uncertainty sets of the following form:
\begin{equation}\label{eq:delta-ad-avg} 
\begin{aligned}
  \Delta_{\text{ad}}(\bb{\rho}) = \bigg\{ \bb{\delta}\in[0,1]^n :\ 
  & \sum_{e=1}^n
  \frac{v_e\bb{\tilde{\rho}}_e^p\delta_e}{|\Omega|} = D_1,\\
  &\sum_{e=1}^n \frac{v_e(\delta_e-0.4)^2}{|\Omega|}	 \leq D_2 \bigg\},
\end{aligned}
\end{equation}
or similarly
\begin{equation}\label{eq:delta-ad-avg2}
\begin{aligned}
  \Delta_{\text{ad}}(\bb{\rho}) = \bigg\{ \bb{\delta}\in[0,1]^n :\ 
  & \sum_{e=1}^n
  \frac{v_e\bb{\tilde{\rho}}_e^p\delta_e}{|\Omega|} = D_1,\\
  &\sum_{e=1}^n \frac{v_e(\bb{\tilde{\rho}}_e^p\delta_e-0.4)^2}{|\Omega|} \leq D_2 \bigg\},
\end{aligned}
\end{equation}
For an interpretation of these sets we refer to section 5.2. From a theoretical point of view, for fixed $\bb{\rho}$, either of the sets is compact and convex and thus no special attention is required throughout this section.
\end{remark}

Now, in order to analyse the structure of the minimax problem we rephrase problem \eqref{eq:problem-state} equivalently as:
\begin{align}\label{eq:problem-energy}
  \min\limits_{\bb{\rho}\in U_{\text{ad}}}\max\limits_{\bb{\delta}\in
  \Delta_{\text{ad}}, \bb{u}\in \RR^m}2\bb{f}^\top
\bb{u}-\sum_{e=1}^n\bb{u}^\top
\bb{K}_e(\bb{\rho},\bb{\delta}) \bb{u}.
\end{align}
Along with \eqref{eq:problem-energy} we study the problem
\begin{align}\label{eq:problem-energy-reg}
  \min\limits_{\bb{\rho}\in U_{\text{ad}}}\max\limits_{\bb{\delta}\in
  \Delta_{\text{ad}}, \bb{u}\in \RR^m} \mathcal{J}_\epsilon(\bb{\rho},\bb{\delta},\bb{u}).
\end{align}
where 
\begin{equation}\label{eq:J-epsilon}
\mathcal{J}_\epsilon(\bb{\rho},\bb{\delta},\bb{u}) =2\bb{f}^\top
\bb{u}-\sum\limits_{e=1}^n\bb{u}^\top
\bb{K}_e(\bb{\rho},\bb{\delta}) \bb{u}-\frac{\epsilon}{2}\|\bb{\delta}\|^2,
\end{equation}
is a regularized variant of the objective function with $\epsilon>0$ being a small positive number causing the last term to act as
a Tikhonov type regularization. We are now able to state the following important result.
\begin{theorem}\label{thm:concavity-J-eps}
The functional $\mathcal{J}_\epsilon$ is concave w.\,r.\,t.
$(\bb{\delta},\bb{u})$ for $\epsilon=0$ and strictly concave for $\epsilon>0$.
\end{theorem}
\begin{proof}
For the convenience of the reader, we first introduce the local
stiffness matrices for a unit material tensor from (\ref{eq:unit-material-tensor}) as
\begin{equation}
\hat{\bb{K}}_e=\sum_{l=1}^{n_{ig}}\bb{B}_{el}^\top
\bb{C}\bb{B}_{el}.
\end{equation}
With $E(\delta) := \left(\frac{1-\delta}{E_0}+\frac{\delta}{E_D}\right)^{-1}$ we then have
\begin{equation}
\bb{K}=\sum_{e=1}^n\bb{K}_e=\sum_{e=1}^n
\tilde{\rho}_e^p E(\delta_e) \hat{\bb{K}}_e.
\end{equation}
and for $i,j=1,\ldots n$ it follows from straightforward calculations that
\begin{align}
	\frac{\partial^2}{\partial \delta_i \partial\delta_j}\mathcal{J}
	&=\begin{cases}
	\tilde{\rho}_i^p (E(\delta_i))^3
	\left( \frac{1}{E_D}-\frac{1}{E_0} \right)^2 \bb{u}^\top\hat{\bb{K}}_i\bb{u} &
	i=j,\\
	0 & \mbox{else,}
	\end{cases}\\
	\frac{\partial^2}{\partial \delta_i \partial \bb{u}}\mathcal{J}
	&=
	\tilde{\rho}_i^p (E(\delta_i))^2
	\left( \frac{1}{E_D}-\frac{1}{E_0} \right) \hat{\bb{K}}_i\bb{u},\\
	\frac{\partial^2}{ (\partial \bb{u})^2}\mathcal{J}&=\bb{K}.
\end{align}
Now we are able to show the (strict) concavity by demonstrating the negative (semi-)definiteness of the Hessian of $\mathcal{J}_\epsilon$:
\begin{equation}\label{eq:hessian-semi-definite}
\begin{aligned}
    \begin{pmatrix}\bb{x}\\\bb{y}\end{pmatrix}^\top&\frac{\partial^2\mathcal{J}_\epsilon}{
    (\partial ({\bb{\delta},\, \bb{u}}))^2}\begin{pmatrix}\bb{x}\\ \bb{y}\end{pmatrix}\\
    &= \bb{x}^\top \frac{\partial^2\mathcal{J}_\epsilon}{\left(\partial
    \bb{\delta}\right)^2} \bb{x} + 2
    \bb{y}^\top\frac{\partial^2\mathcal{J}_\epsilon}{\partial \bb{\delta} \partial
    \bb{u}}\bb{x} + \bb{y}^\top \frac{\partial^2\mathcal{J}_\epsilon}{ (\partial
    \bb{u})^2} \bb{y} \\
    &= -\sum_{e=1}^n x_e^2\tilde{\rho}_e^p (E(\delta_e))^3
    \left( \frac{1}{E_D}-\frac{1}{E_0} \right)^2 \bb{u}^\top\hat{\bb{K}}_e\bb{u} \\
	& \quad + 2\sum_{e=1}^n x_e
	\tilde{\rho}_e^p (E(\delta_e))^2
	\left( \frac{1}{E_D}-\frac{1}{E_0} \right)\bb{y}^\top \hat{\bb{K}}_e\bb{u} \\
	& \quad - \bb{y}^\top \bb{K} \bb{y} -\epsilon \|\bb{x}\|^2\\
	&=-\sum_{e=1}^n	\tilde{\rho}_e^p E(\delta_e)
	\left\{ \bb{v}^\top	\hat{\bb{K}}_e \bb{v}
	-2\bb{v}^\top \hat{\bb{K}}_e \bb{y}
	+ \bb{y}^\top \hat{\bb{K}}_e \bb{y}\right\} \\
	& \quad -\epsilon \|\bb{x}\|^2\\
	&=-\sum_{e=1}^n	\tilde{\rho}_e^p E(\delta_e) \left\{ 
	\left[\bb{y} - \bb{v} \right]^\top \hat{\bb{K}}_e \left[\bb{y} - \bb{v} \right]
	\right\} -\epsilon \|\bb{x}\|^2 \\
	& \begin{cases}
	\leq 0 & \epsilon=0\\
	<0 & \epsilon>0
	\end{cases}  \quad \forall
	\begin{pmatrix}\bb{x}\\\bb{y}\end{pmatrix}\in \RR^{n+m}
\end{aligned}
\end{equation}
where $\bb{v} = x_e E(\delta_e)\left( \frac{1}{E_D}-\frac{1}{E_0} \right) \bb{u}$. The latter inequality holds because of the local stiffness matrices $\hat{\bb{K}}_e $ being positive semi-definite,
$\tilde{\rho}_e^p E(\delta_e)
> 0$ and $\sum_{e=1}^n \tilde{\rho}_e^p E(\delta_e) \hat{\bb{K}}_e$ positive definite.
\end{proof}

A direct consequence of \thref{thm:concavity-J-eps} is that the inner maximization problem in \eqref{eq:problem-state} has -- for either of the choices of the uncertainty set outlined above -- a unique solution for any feasible choice of of the pseudo-density $\bb{\rho}$ provided that the regularization parameter $\epsilon$ is positive. Moreover the computation of this maximizer is numerically tractable.
In other words, the marginal function
\begin{equation}\label{eq:marginal} 
F_\epsilon(\bb{\rho}) := \max\limits_{\bb{\delta}\in
  \Delta_{\text{ad}}, \bb{u}\in \RR^m} \mathcal{J}_\epsilon(\bb{\rho},\bb{\delta},\bb{u})
\end{equation}
is well defined and can be efficiently evaluated. Moreover, as a maximum over infinitely many linear functions $F_{epsilon}$ is even continuous w.r.t. the design variable $\bb{\rho}$.

\subsection{Discussion of the inverse interpolation model} \label{sec:dis_interpol}
As outlined previously the inverse interpolation formula \eqref{eq:matdelta} has the strong advantage that the resulting adversarial problem 
\begin{align}\label{eq:advarsarial}
\max\limits_{\bb{\delta}\in
  \Delta_{\text{ad}}, \bb{u}\in \RR^m} \mathcal{J}_\epsilon(\bb{\rho},\bb{\delta},\bb{u})
\end{align}
is concave and can thus be solved to global optimality with a reasonable effort. On the other hand, from an application point of view using the standard interpolation model \eqref{eq:damage-model} appears to be much more natural.

Motivated by this, throughout this section, we would like to discuss the impact of using the interpolation model \eqref{eq:matdelta} as an approximation of the linear interpolation model \eqref{eq:damage-model} as well as a connection of the two models through a continuation scheme for the uncertainty variable $\bb{\delta}$.

We start with a closer look at formula \eqref{eq:matdelta}. As we have seen in the previous section, the properties of \eqref{eq:matdelta} were essential for the compliance to be concave in $\bb{\delta}$. In
fact, this was one of the goals Stolpe and Svanberg wanted to achieve with the suggestion of an alternative penalization model in topology
optimization, see \cite{stolpe2001alternative}. This interpolation scheme later became better known by the acronym RAMP for for Rational Approximation of Material Properties. Using the RAMP model to prescribe the effect of material uncertainty on the effective material stiffness in an $e$, we may write:
\begin{equation}\label{eq:ramp}
E^{\text{eff}}_e=E_e(\delta_e)=E_D+\frac{1-\delta_e}{1+q\delta_e}(E_0-E_D).
\end{equation}
Here, the parameter $q\geq 0$ serves is the RAMP penalization parameter. It follows directly from
\cite{stolpe2001alternative} that using this interpolation scheme in a linear elastic framework, the compliance is concave for all $q\leq (E_0-E_D)/E_D$. Moreover, it can be
easily checked that in the case $q = (E_0-E_D)/E_D$ the interpolation scheme \eqref{eq:ramp} 
coincides with our inverse interpolation scheme \eqref{eq:matdelta}, i.e. we can represent \eqref{eq:matdelta} in a RAMP fashion.
\begin{figure}[ht]
\centering
\begin{tikzpicture}[scale=.8]
  \node at (-.6,-.2)
  {\includegraphics[height=.75\columnwidth]{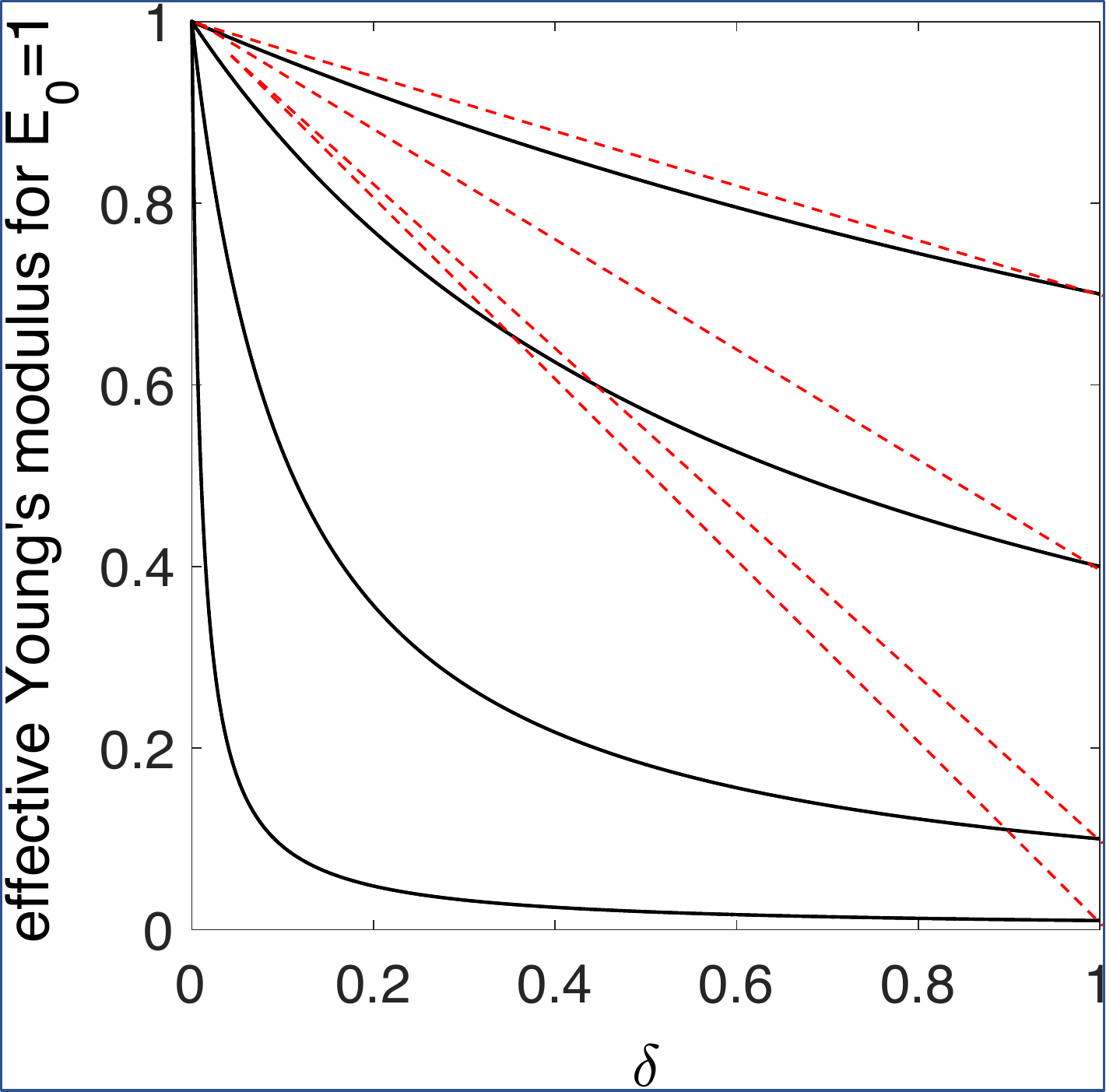}};
  \node[text width=2.5cm] at (5.0,1.6) {$E_D=0.7$};
  \node[text width=2.5cm] at (5.0,-.4) {$E_D=0.4$};
  \node[text width=2.5cm] at (5.0,-2.3) {$E_D=0.1$};
  \node[text width=2.5cm] at (5.0,-2.9) {$E_D=0.01$};
\end{tikzpicture}
 \caption{Visualization of the direct and the inverse interpolation model for different material parameters $E_0, E_D$
 \label{fig:visualmat}}
\end{figure}
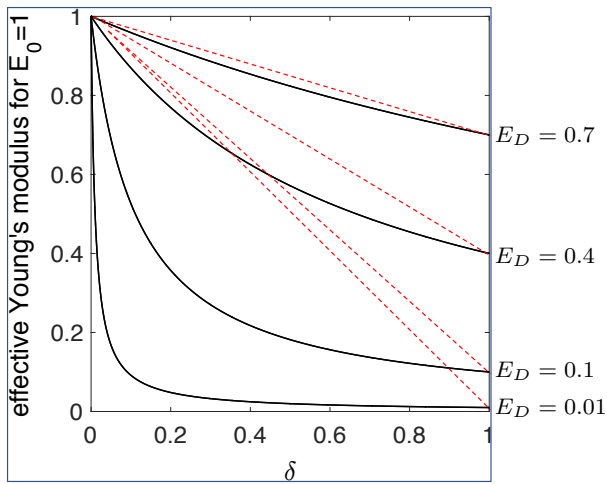
Examples of the effective material stiffness computed by \eqref{eq:matdelta} for  different
values of ${\delta}$ and $E_D$ are shown in \figref{fig:visualmat}. 
It is seen that for an increasing difference between $E_0$ and $E_D$, the penalization
parameter $q=(E_0-E_D)/E_D$ leads to increasing penalization of intermediate values strictly between 0 and 1. However, as the compliance is
maximized w.r.t. $\bb{\delta}$ (see \eqref{eq:advarsarial}), this means rewarding intermediate values if a linear constraint is used to bound the total material uncertainty as in \eqref{eq:delta-ad} or \eqref{eq:delta-ad-rho}. 

Now, in order to compare the effect of the linear and the inverse interpolation scheme using the same linear total uncertainty bound, we can sketch the effective material properties obtained in both cases. \figref{fig:external-approx} demonstrates this for a two element setup with a total uncertainty bound $\delta_1+\delta_2 \leq 1$ and material data $E_0=1$ and $E_D=0.7$. 
\begin{figure}[!ht]
\centering
\subfigure[admissible uncertainties]
 {\includegraphics[width=.495\columnwidth]{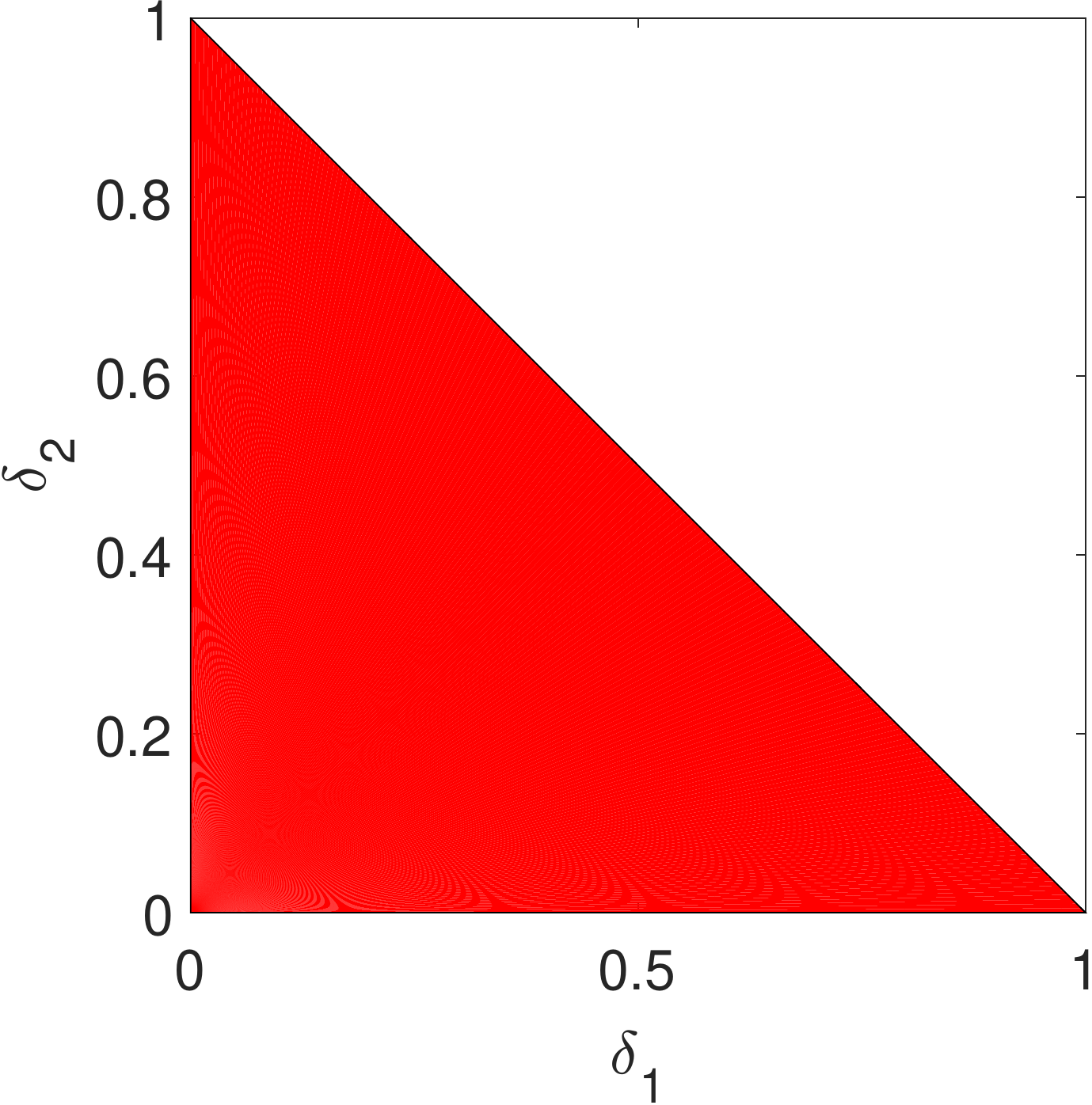}}\hfill
 \subfigure[admissible Young's moduli (linear model)]
 {\includegraphics[width=.495\columnwidth]{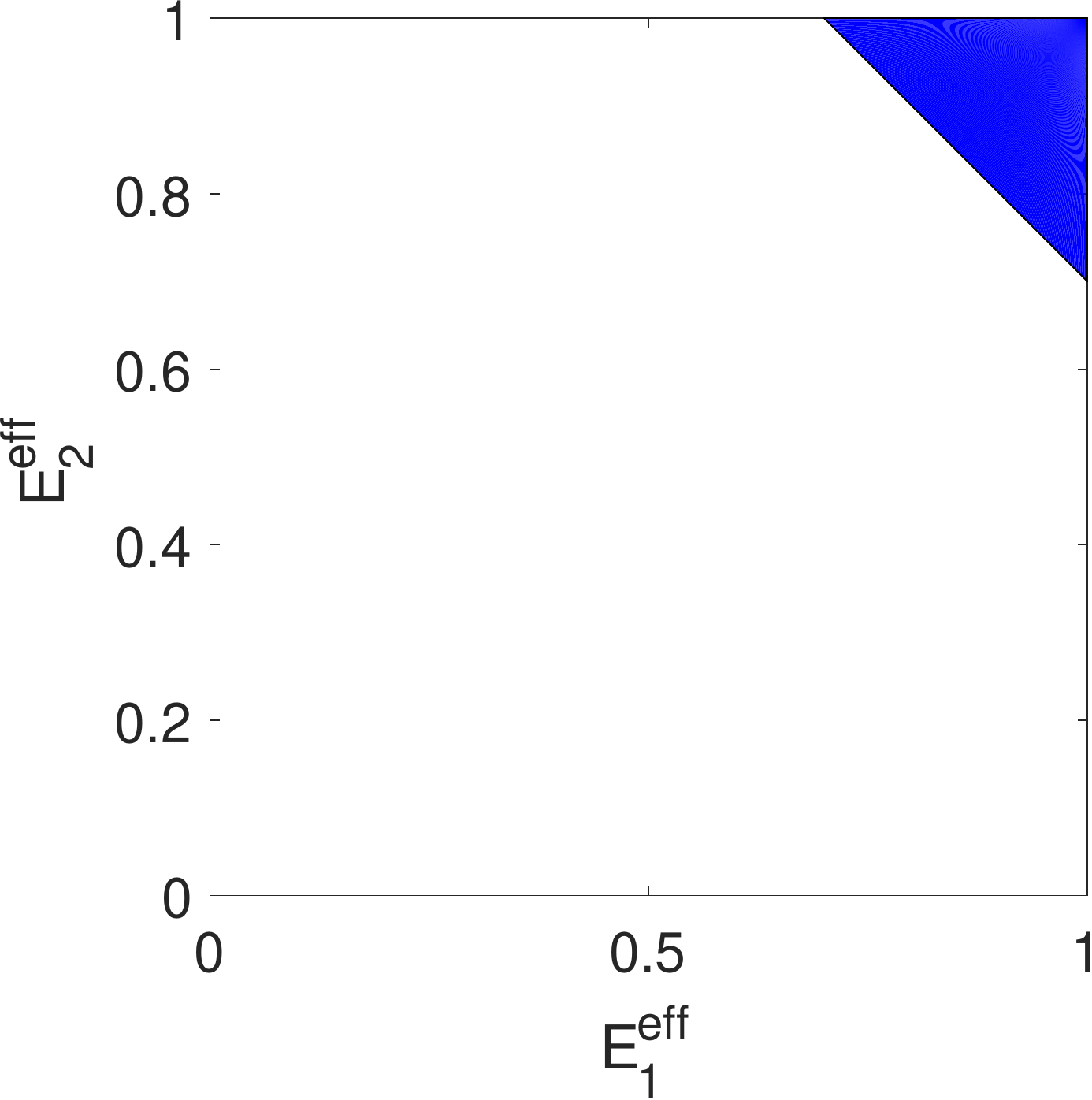}}
\subfigure[admissible Young's moduli (inverse model)]
 {\includegraphics[width=.495\columnwidth]{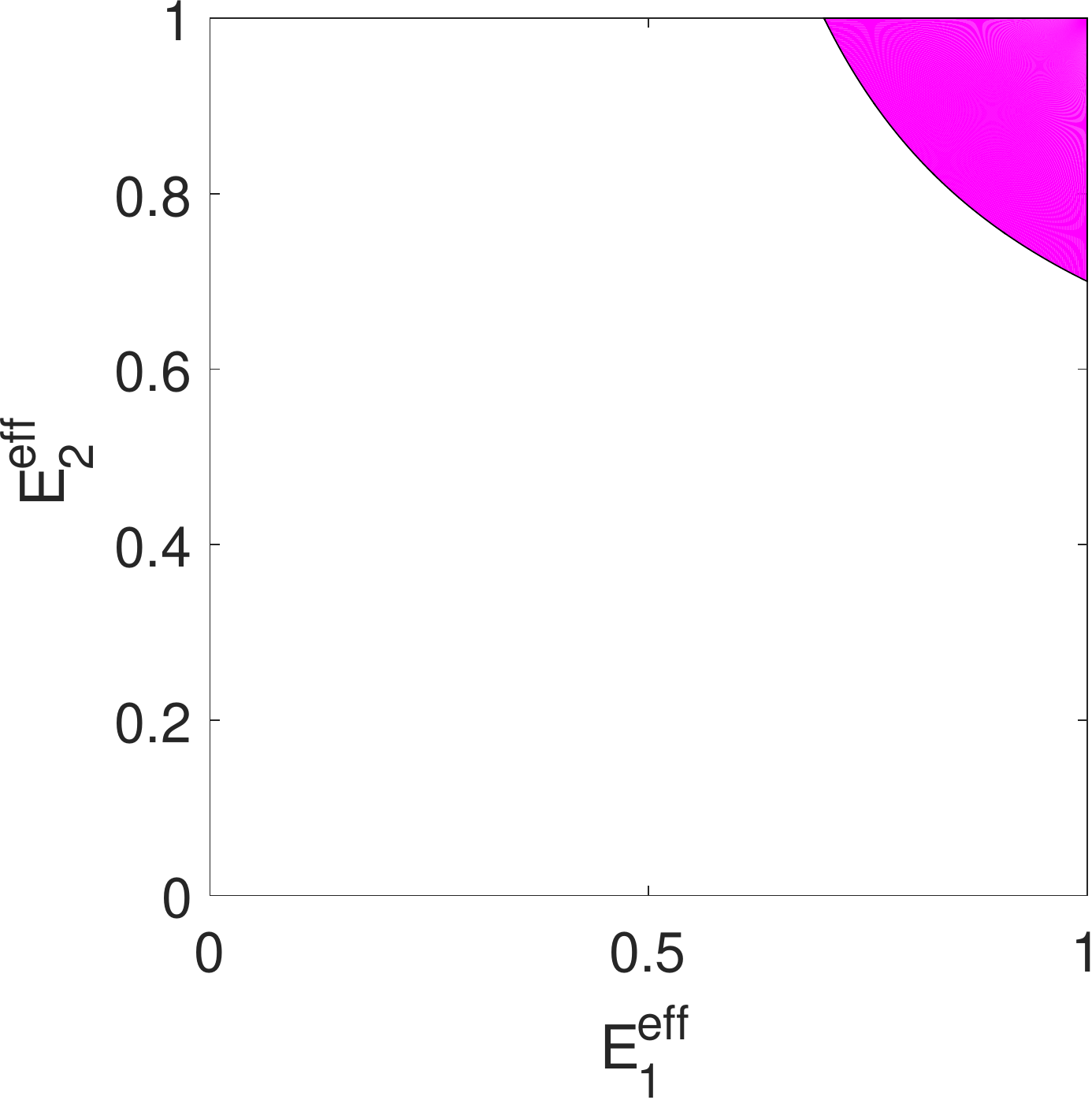}}\hfill
\subfigure[comparison of effective uncertainty sets]
 {\includegraphics[width=.495\columnwidth]{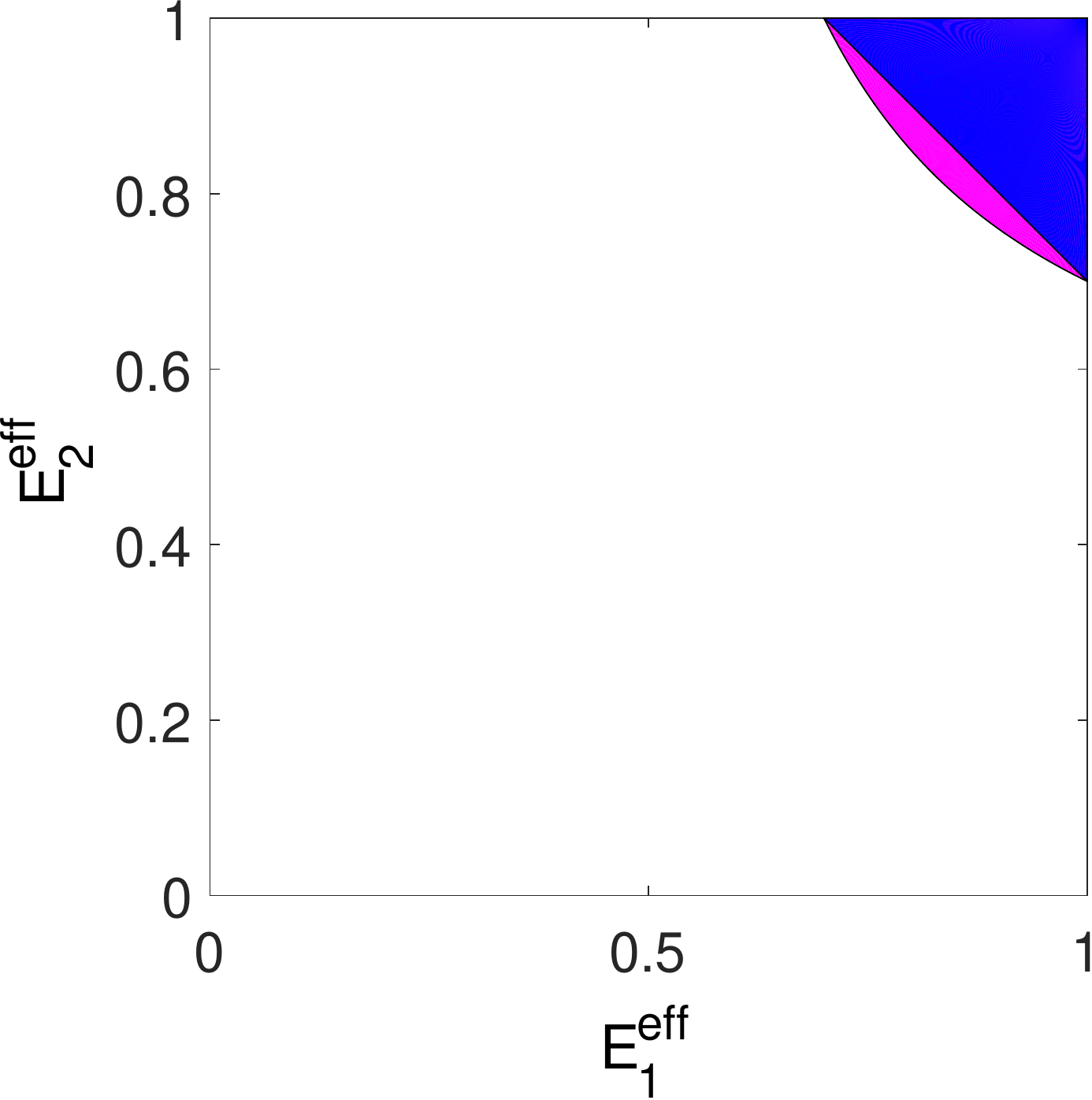}}
 \caption{Comparison of the linear and inverse stiffness interpolation models
 for a domain with two elements and the total uncertainty constraint
 $\delta_1+\delta_2\leq 1$.
 \label{fig:external-approx}}
\end{figure}
It is clearly seen that the admissible set obtained using the inverse interpolation model (magenta set) is a relatively tight outer approximation of the admissible set obtained using the linear interpolation model (blue set) for this choice of material parameters (see \figref{fig:external-approx}d). This is true in general, if the material deviations characterized by the difference $(E_0-E_D)/E_0$ are sufficiently close to 1. If this is the case, applying model \eqref{eq:matdelta} as an outer approximation of model \eqref{eq:damage-model} in the framework of the robust optimization model \eqref{eq:problem-state} will lead to only slightly more conservative topological solutions. This appears to be a small price, taking the numerical tractability into account. It is noted that $(E_0-E_D)/E_0$ is comparably large all additive manufacturing scenarios discussed in section 5.

On the other hand \figref{fig:inner-outer-approx} reveals that for smaller and smaller $E_D$, it becomes more and more likely that the worst-case effective material (i.e. the maximizer of the compliance function), is close to the black spots highlighted in \figref{fig:inner-outer-approx}.
If this happens, the solution of the adversarial problem \eqref{eq:advarsarial} can be quite far away from the original uncertainty sets (blue sets in \figref{fig:inner-outer-approx}) and thus lead to an over-conservative solution. As a consequence, from an application point of view it might be more attractive to work directly with the linear interpolation model. As the latter can not be expected to be solved to global optimality for a meaningful size of the problem, we suggest the following heuristic based on the RAMP interpretation of model \eqref{eq:matdelta}: we start with RAMP parameter $q = (E_0-E_D)/E_D$ --  corresponding precisely to the concave variant of the adversarial problem -- and from there slowly decrease the RAMP parameter towards 0, corresponding precisely to the linear interpolation rule. It should be stressed that there is no guarantee that this procedure leads to the global maximizer of \eqref{eq:advarsarial} with linear interpolation model, however it provides a lower bound for the worst case compliance of a given topology. As the inverse model at the same time provides a rigorous upper bound for the same, at least this information can be used to estimate how conservative a computed robust solution is.
Moreover numerical experiments in section 5.3 will reveal that much more reasonable results are obtained compared to a direct application of the uncertainty model \eqref{eq:damage-model}. 

We finally note that one may also think about inner approximations (green sets in \figref{fig:inner-outer-approx}) instead of outer approximations, but then the true worst-case might be missed, when solving the adversarial problem and thus the term worst-case looses its meaning. This is why we did not further pursue this direction in this paper.

\begin{figure}[!ht]
\centering
\subfigure[$E_D=0.7$]
 {\includegraphics[width=.47\columnwidth]{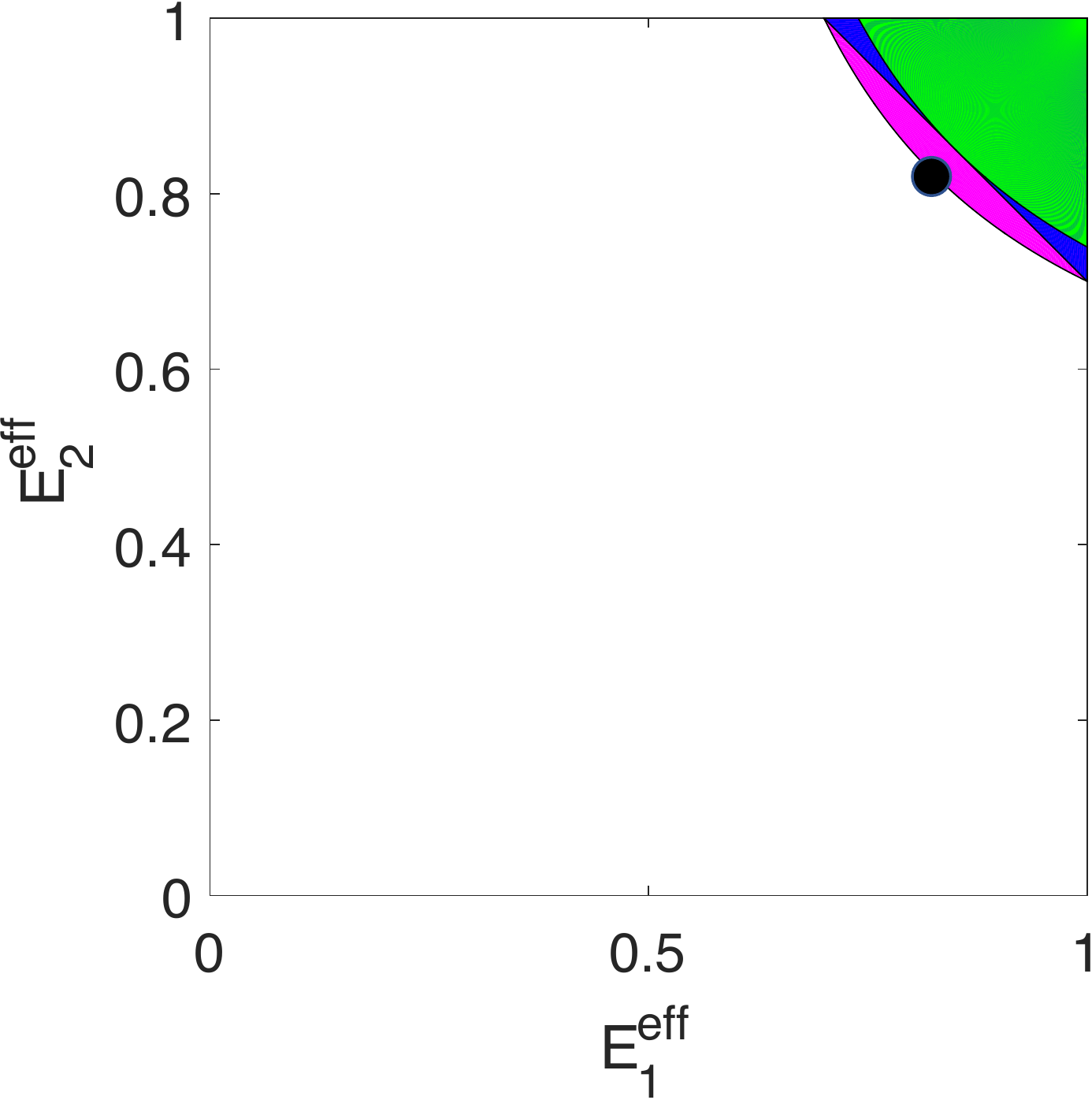}}\hfill
 \subfigure[$E_D=0.5$]
 {\includegraphics[width=.47\columnwidth]{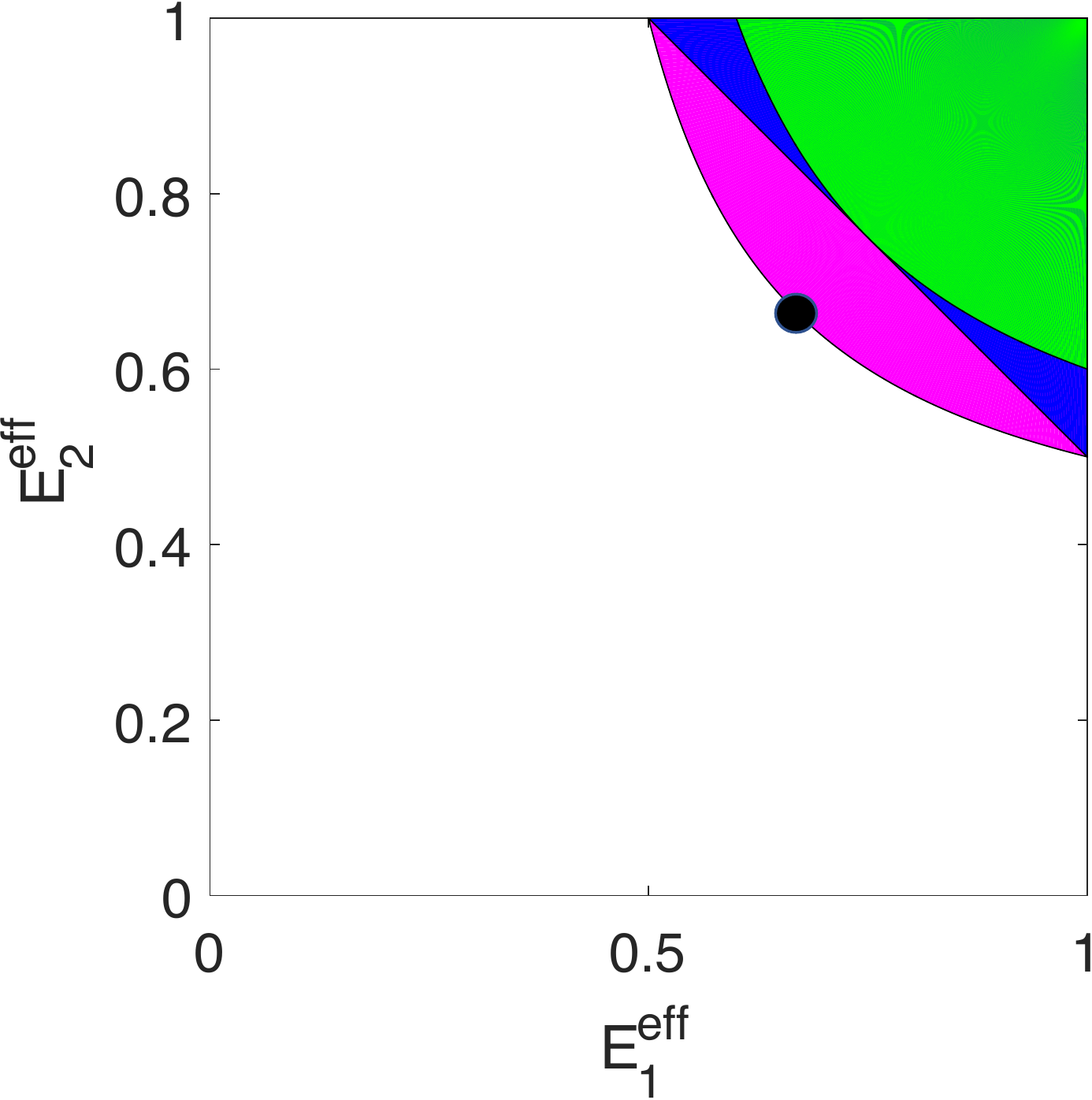}}
\subfigure[$E_D=0.2$]
 {\includegraphics[width=.47\columnwidth]{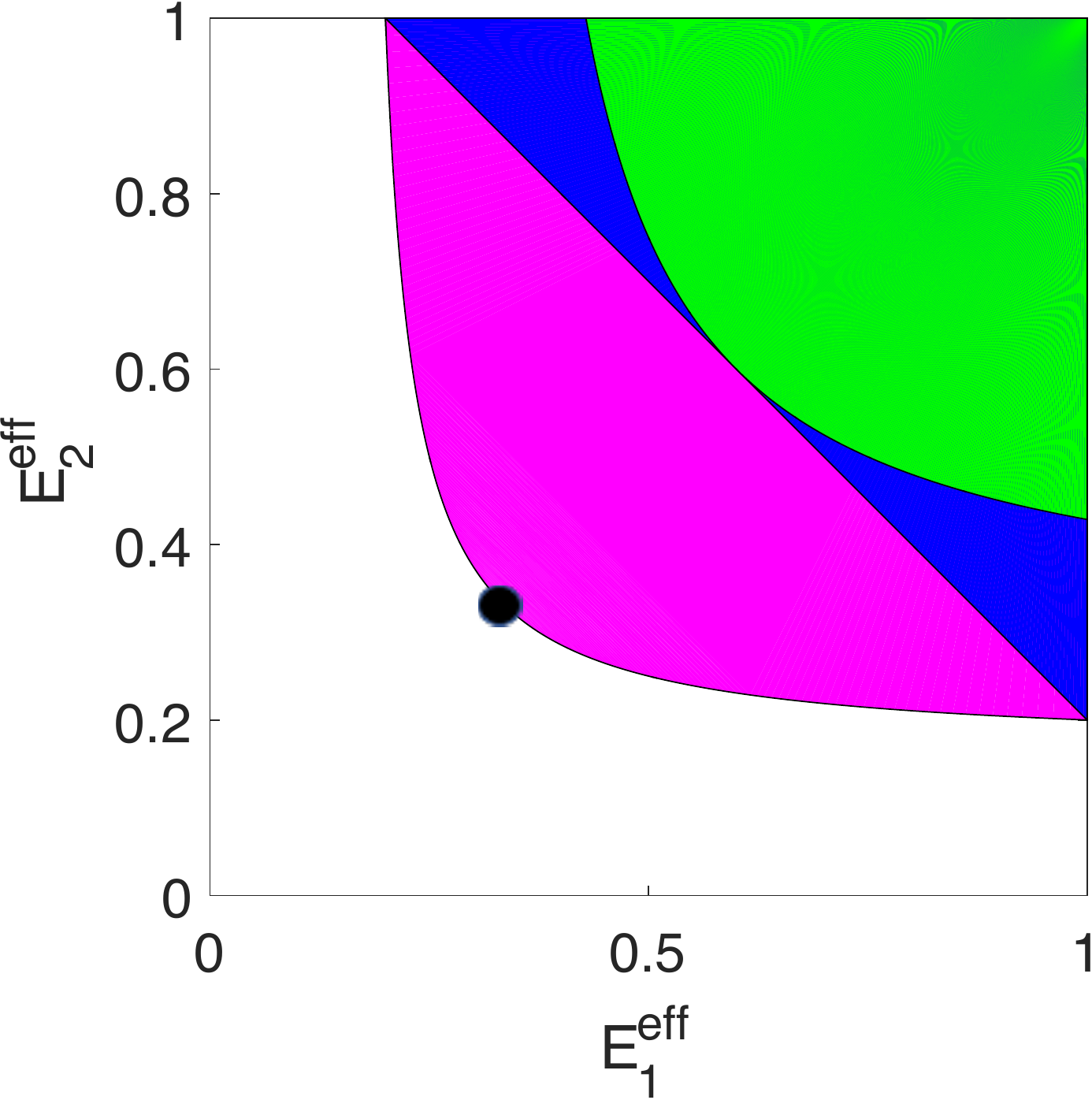}}\hfill
\subfigure[$E_D=0.01$]
 {\includegraphics[width=.47\columnwidth]{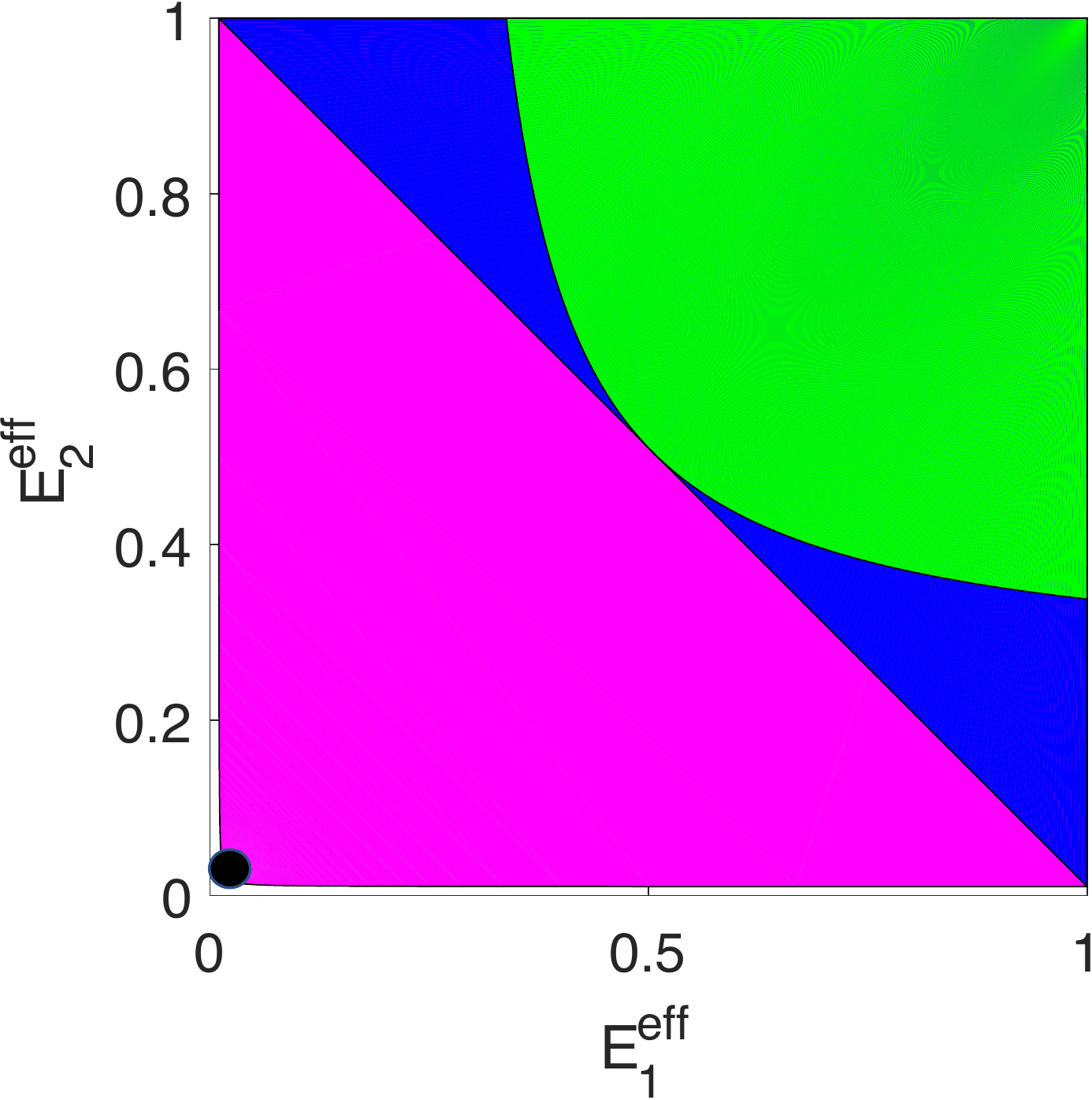}}
 \caption{Comparison of the linear and inverse stiffness interpolation models
 for a domain with two elements and a total uncertainty constraint of
 $\delta_1+\delta_2\leq1$. The images show the admissible effective materials for the linear model
 (blue) and an outer (magenta) and inner (green) approximation using the inverse model. Note the small distance of the highlighted points in the magenta set to the origin for small values of $E_D$.
 \label{fig:inner-outer-approx}} 
\end{figure}
\section{Obtaining derivative information}
In \secref{sec:optproblem} we have seen that the marginal function $F_\epsilon(\bb{\rho})$ is well-defined and continuous. Taking additionally the compactness of the set $U_{\text{ad}}$ (see \eqref{eq:U-ad}) into account it is clear that the problem 
\begin{align}\label{eq:optval_func}
  \min\limits_{\bb{\rho}\in U_{\text{ad}}}F_\epsilon(\bb{\rho}),
\end{align}
which is fully equivalent to the regularized problem \eqref{eq:problem-energy-reg}, has a solution. As the objective $F_\epsilon$ depends only on $\bb{\rho}$ in an explicit manner, we may consider using standard, gradient-based optimzation solvers for the numerical solution of \eqref{eq:optval_func}. This is our motivation to analyse differentiability properties of $F_\epsilon$ in the sequel.

\subsection{The $\epsilon$-regularized minimax approach}\label{sec:minimax}
In order to investigate the differentiability properties of $F_\epsilon$ we make use of results from literature formulated for marginal functions of the type
\begin{equation}\label{eq:general_minmax}
\psi(\bb{x})=\max\limits_{\bb{y}\in Y(\bb{x})}\mathpzc{j}(\bb{x},\bb{y}),
\end{equation}
where $\mathpzc{j}:\RR^l\times\RR^k\to \bb{\RR}$ and $Y(\cdot)$ is a
point-to-set map from $\RR^k$ to $\RR^l$. Moreover we assume that the set $Y(x)$ (which is in the context of our paper the uncertainty set) is given in the following form: 
\begin{equation*}
Y(\bb{x})= \left\{ \bb{y}\in Y : G(\bb{x},\bb{y})\leq \bb{0},  H(\bb{x},\bb{y}) = \bb{0} \right\},
\end{equation*}
where $G:\RR^l\times\RR^k\to\RR^s, (x,y) \mapsto (g_1(x,y), \ldots, g_k(x,y))$ and $H:\RR^l\times\RR^k\to\RR^q, (x,y) \mapsto (h_1(x,y), \ldots, h_q(x,y))$ are continuously differentiable in both variables. 

Next we state two well known constraint qualifications for the maximization problem in \eqref{eq:general_minmax}.
\begin{definition}[Mangasarian-Fromowitz constraint
qualification]\label{def:mfcq} 
A feasible point $\bb{\bar{y}}\in Y(\bb{\bar{x}})$ is said to satisfy the
\textit{Mangasarian-Fromowitz constraint qualification} (MFCQ) if the following holds:
\begin{enumerate}
  \item There exists a direction $\bb{d}\in\RR^k$ such that $\nabla_{\bb{y}}H(\bb{\bar{x}},\bb{\bar{y}})^\top\bb{d}=\bb{0}$ and $\nabla_{\bb{y}}g_i(\bb{\bar{x}},\bb{\bar{y}})^\top \bb{d}<0$ for all $i\in
      I(\bb{\bar{x}},\bb{\bar{y}})=\left\{i\in\{1,\ldots,s\} \mid  g_i(\bb{\bar{x}},\bb{\bar{y}}) = 0\right\}$.
  \item The gradients of the equality constraints $\nabla_{\bb{y}}
  h_j,j=1,\ldots,q$ are linearly independent at $(\bb{\bar{x}},\bb{\bar{y}})$.
\end{enumerate}
\end{definition}
\begin{definition}[Linear independence constraint qualification]\label{def:licq}
A feasible point $\bb{\bar{y}}\in Y(\bb{\bar{x}})$ is said to satisfy the
\textit{linear independence constraint qualification} (LICQ), if the system of gradients
$\nabla_{\bb{y}}h_i(\bb{\bar{x}},\bb{\bar{y}}), \ i=1,\ldots,q$,
$\nabla_{\bb{y}}g_j(\bb{\bar{x}},\bb{\bar{y}}), \ j\in I(\bb{\bar{x}},\bb{\bar{y}})$ is linearly independent.
\end{definition}
Note that the LICQ is stronger and in fact implies MFCQ. Furthermore it is well known that the LICQ immediately guarantees the uniqueness of Lagrangian multipliers associated with the constraints $G(\bb{x},\bb{y})\leq \bb{0},  H(\bb{x},\bb{y}) = \bb{0}$.
Moreover we need the following definition.

\begin{definition}
The set $Y(\bar{x})$ is called \textit{uniformly compact} near $\bb{\bar{x}}$ if there is a
neighborhood $N(\bb{\bar{x}})$ of $\bb{\bar{x}}$, such that the closure of
$\bigcup_{\bb{x}\in N(\bb{\bar{x}})}Y(\bb{x})$ is compact.
\end{definition}

After these preliminaries, we can state the following result taken from
\cite{gauvin1982differential}:
\begin{theorem}\label{thm:subdifferential}
If $Y(\bb{\bar{x}})$ is nonempty, $Y(\bb{\bar{x}})$ is uniformly compact near $\bb{\bar{x}}$ and the MFCQ is satisfied at every point $\bb{\bar{y}}\in
M(\bb{\bar{x}})$, then the following inclusion holds true for the generalized gradient in the sense of Clarke:
\begin{align}\label{eq:subdifferential-inclusion}
\partial\psi(\bb{\bar{x}})\subset & \operatorname{conv} \big\{
\bigcup_{\bb{\bar{y}}\in M(\bb{\bar{x}})} \bigcup_{(\bb{\lambda},\bb{\mu})\in
K(\bb{\bar{x}},\bb{\bar{y}})}\nabla_{\bb{x}}\mathpzc{j}(\bb{\bar{x}},\bb{\bar{y}})\nonumber \\ 
& -\bb{\lambda}^\top\nabla_{\bb{x}}G(\bb{\bar{x}},\bb{\bar{y}})
-\bb{\mu}^\top\nabla_{\bb{x}}H(\bb{\bar{x}},\bb{\bar{y}}) \big\}.
\end{align}
Here $M(\bb{\bar{x}})$ is the set of maximizers for a given $\bb{\bar{x}}$, i.~e.
\begin{equation*}
M(\bb{\bar{x}})= \left\{ \bb{y}\in Y(\bb{\bar{x}}) : 
\psi(\bb{\bar{x}}) = \mathpzc{j}(\bb{\bar{x}},\bb{y}) \right\}
\end{equation*}
and $K(\bb{\bar{x}},\bb{\bar{y}})$ is the set of optimal Lagrange multipliers
\begin{equation*}
\begin{aligned}
K(\bb{\bar{x}},\bb{\bar{y}})= \bigg\{ &(\bb{\lambda},\bb{\mu})
\in\RR^s\times\RR^q :\\
&\nabla_{\bb{y}}\mathpzc{j}(\bb{\bar{x}},\bb{\bar{y}})
-\bb{\lambda}^\top\nabla_{\bb{y}}G(\bb{\bar{x}},\bb{\bar{y}})
-\bb{\mu}^\top\nabla_{\bb{y}}H(\bb{\bar{x}},\bb{\bar{y}})=\bb{0},\\
&\lambda_i g_i(\bb{\bar{x}},\bb{\bar{y}})=0,\ i=1,\ldots,s,
\bb{\lambda}\geq \bb{0}
\bigg\}.
\end{aligned}
\end{equation*}
If, in addition, the LICQ holds at every point $\bb{\bar{y}}\in
M(\bb{\bar{x}})$
\begin{equation}\label{eq:subdifferential}
\begin{aligned}
\partial\psi(\bb{\bar{x}})= \operatorname{conv} \bigg\{
\bigcup_{\bb{\bar{y}}\in M(\bb{\bar{x}})}
&\nabla_{\bb{x}}\mathpzc{j}(\bb{\bar{x}},\bb{\bar{y}})
-\bb{\bar{\lambda}}^\top\nabla_{\bb{x}}G(\bb{\bar{x}},\bb{\bar{y}})\\
&-\bb{\bar{\mu}}^\top\nabla_{\bb{x}}H(\bb{\bar{x}},\bb{\bar{y}}) \bigg\}.
\end{aligned}
\end{equation}
where $(\bb{\bar{\lambda}},\bb{\bar{\mu}})$ are the unique multipliers from
$K(\bb{\bar{x}},\bb{\bar{y}})$.
\end{theorem}
Finally, the following theorem is taken from \cite{clarke1975generalized}:
\begin{theorem}\label{thm:subdifferential-singleton}
The following are equivalent:
\begin{enumerate}
  \item $\partial \psi(\bb{x})= \left\{ \bb{\zeta} \right\}  $, i.\,e. the
  generalized gradient reduces to a singleton.
  \item $\nabla \psi(\bb{x}) $ exists, $\nabla \psi(\bb{x})=\bb{\zeta} $, and
  $\nabla\psi$ is continuous at $\bb{x}$ relative to the set on which it is
  defined.
\end{enumerate}
\end{theorem}
Based on these general results, we can state the differentiability result for  $F_\epsilon$:
\begin{theorem}\label{thm:differentiability-minimax}
Let
\begin{equation}\label{eq:psi-connected}
\begin{aligned}
F_\epsilon(\bb{\rho}):&= \max\limits_{\bb{\delta}\in
  \Delta_{\text{ad}(\bb{\rho})}, \bb{u}\in \RR^m}
  \mathcal{J}_\epsilon(\bb{\rho},\bb{\delta},\bb{u})\\
  &=\max\limits_{\bb{\delta}\in
  \Delta_{\text{ad}(\bb{\rho})}, \bb{u}\in \RR^m}2\bb{f}^\top \bb{u}-\sum\limits_{e=1}^n\bb{u}^\top
\bb{K}_e(\bb{\rho},\bb{\delta}) \bb{u}-\frac{\epsilon}{2}\|\bb{\delta}\|^2,
\end{aligned}
\end{equation}
with $\epsilon > 0$ and $\Delta_{\text{ad}}(\bb{\rho})$ is given by formula \eqref{eq:delta-ad-rho}. 
Let further $(\bb{\bar{\delta}},\bb{\bar{u}})$ denote the unique maximizer
with $F_\epsilon(\bb{\bar{\rho}})=\mathcal{J}_\epsilon(\bb{\bar{\rho}},\bb{\bar{\delta}},\bb{\bar{u}})$
for $\bb{\bar{\rho}}\in U_{\text{ad}}$ and assume that the LICQ condition is satisfied at every point $(\bb{\bar{\delta}},\bb{\bar{u}})\in M(\bb{\bar{\rho}})$.

Then the following holds: 
$F_\epsilon$ is continuously differentiable at
$\bb{\bar{\rho}}$ and
\begin{equation}\label{eq:gradient}
\nabla F_\epsilon(\bb{\bar{\rho}})=\left[\bb{\bar{u}}^\top\frac{\partial
\bb{K}(\bb{\bar{\rho}},\bb{\bar{\delta}})}{\partial\rho_e}\bb{\bar{u}}\right]_{e=1}^n
-\bar{\lambda}\nabla_{\bb{\rho}} \left(\sum_{e=1}^n
  v_e \tilde{\bar{\rho}}_e^p\delta_e/|\Omega| \right).
\end{equation}
\end{theorem}
\begin{remark}
Although we can not prove that LICQ holds in general, we would like to remark that LICQ is expected to hold in any practically relevant situation. To see this, we note
that the only way that LICQ will not hold at
some point $\bb{\bar{\rho}}$ is that the worst-case solution $\bb{\bar{\delta}}$ is fully
binary and $\sum_{e:\bar{\delta}_e=1}v_e\tilde{\bar{\rho}}_e/|\Omega|=D$.
This is very unlikely for any $D<|\Omega|$. 
Furthermore, as $\bb{\bar{\delta}}$ is mostly large where $\bb{\tilde{\rho}}^p$ has value $1$, $D$ may be chosen such
that $\sum_{e:\bar{\delta}_e=1}v_e/|\Omega|$ cannot equal $D$ (e.\,g. for a
uniform finite element size s.\,t. $n\cdot D\notin \mathbb{N}$). Finally, in practice, no
setting could be found which resulted in a 0-1 distribution of $\bb{\delta}$. 
We furthermore remark that analogous results can be proven, if instead of formula \eqref{eq:delta-ad-rho}, one of the formulae \eqref{eq:delta-ad}, \eqref{eq:delta-ad-avg} or \eqref{eq:delta-ad-avg2} are used.
\end{remark}
\begin{proof}[\thref{thm:differentiability-minimax}]
We want use \thref{thm:subdifferential}. We have
$Y(\bb{\rho})={\Delta_{\text{ad}}(\bb{\rho}) \times \RR^m}$ and
$\mathpzc{j}=\mathcal{J}_\epsilon(\bb{\rho},\bb{\delta},\bb{u})$. 
While $Y(\bb{\rho})$ is nonempty and $\Delta_{\text{ad}}( \bb{\rho})$ is uniformly compact, $\RR^m$ is 
evidently unbounded. However, as in the optimum we have
$\bb{u}=\bb{K}^{-1}(\bb{\rho},\bb{\delta})\bb{f}$ with $\bb{K}$ uniformly positive
definite, the optimal $\bb{u}$ is indeed uniformly bounded. Thus, we can modify
the admissible set of $\bb{u}$ to include artificial lower and upper bounds
without changing the solution of the optimization problem. Now, as we assume LICQ to hold, from 
\thref{thm:concavity-J-eps} and \thref{thm:subdifferential-singleton} we directly obtain the continuous differentiability of $F_\epsilon$. Finally, the gradient formula \eqref{eq:gradient}
follows directly from the general formula \eqref{eq:subdifferential}.

\end{proof}
\begin{remark}
In principal, we can now solve problem \eqref{eq:general_minmax} by our favorite gradient based optimization solver, where in each iteration we solve a maximization problem to evaluate $F_\epsilon$ and apply formula \eqref{eq:gradient} to evaluate the associated gradient. However, this will only work well in practice, if the inner maximization problem is solved with high accuracy in each iteration. Moreover, in order to obtain a good approximation of the gradient, also tight numerical approximations of the Lagrangian multipliers associated with the constraints of the inner problem have to be accessible. Taking these observations as well as the sparsity structure of the inner problem into account it is advisable to use a second order method for the evaluation of $F_\epsilon$. As a matter of fact, interior point methods belong to the most efficient second order optimization methods, in particular if the problem to be solved belongs to the class of convex minimization problems. As the maximization problem required to evaluate $F_{\epsilon}$ can be easily turned into a problem of this kind, we suggest to use an interior point method also here. Still the problem remains that the precise evaluation of $F_\epsilon$ can be numerically quite demanding. As a remedy, in the following section, we suggest an alternative regularization scheme which is more tightly connected to interior point methods and thus allows in general for a more efficient solution of the regularized inner maximization problem.
\end{remark}

\subsection{Regularization using a barrier method}\label{sec:smoothed} 
In this section, we study the approach of using a barrier term for the
regularization of the minimax problem. This will allow to show
differentiability without using the minimax theorems of the
previous section and, more importantly, turn out to be superior w.~r.~t. computational efficiency. \\
We consider (\ref{eq:problem-energy}) as a specific bilevel optimization
problem as follows:
\begin{align}\label{eq:problem-bilevel1}
\left\{
\begin{array}{l}
	\min\limits_{\bb{\rho}\in
	U_{\text{ad}},\bb{\delta}\in \RR^n, \bb{u} \in \RR^m}
\mathcal{J}_0(\bb{\rho},\bb{\delta},\bb{u}) = 2\bb{f}^\top\bb{u}-\sum\limits_{e=1}^n\bb{u}^\top
\bb{K}_e(\bb{\rho},\bb{\delta})\bb{u} \\
	\text{s.t.: } ({\bb{\delta},\, \bb{u}}) \in
	\underset{{\bb{\delta}\in\Delta_{\text{ad}},\, \bb{u}\in
	\RR^m}}{\operatorname{argmax}}\ 2\bb{f}^\top
	\bb{u}-\sum_{e=1}^n\bb{u}^\top \bb{K}_e(\bb{\rho},\bb{\delta}) \bb{u}.
\end{array}
\right.
\end{align}
Furthermore, we write the uncertainty set in the
generic form
\begin{equation}\label{eq:delta-generic}
  \Delta_{\text{ad}} = \left\{ \bb{\delta}\in\RR^n :
  0\leq \delta_e \leq 1,\ e=1,\ldots,n,\ g(\bb{\rho},\bb{\delta}) \leq 0
  \right\},
\end{equation}
where
\begin{equation}\label{eq:vol-delta}
g(\bb{\rho},\bb{\delta})=\sum_{e=1}^n v_e \delta_e/|\Omega|-D \leq 0
\end{equation}
for the set (\ref{eq:delta-ad}) and 
\begin{equation}\label{eq:vol-delta-rho}
g(\bb{\rho},\bb{\delta})=\sum_{e=1}^n
v_e \tilde{\rho}_e^p\delta_e/|\Omega|-D \leq 0.
\end{equation}
for (\ref{eq:delta-ad-rho}). Note that again we may alternatively use the uncertainty sets \eqref{eq:delta-ad-avg} or \eqref{eq:delta-ad-avg2} with only minor modifications.\\
The idea is now to replace the lower-level (or inner) optimization problem by
its optimality conditions and solve the upper-level (or outer) optimization
problem as a constrained problem with these optimality conditions given as
constraints. We remark that for fixed ${\bb{\rho}}$ the uncertainty sets above are affine linear and thus no constraint qualification is required. Thus we may write the optimality conditions
of the lower-level optimization problem as follows: 
\begin{equation}\label{eq:kkt}
\begin{aligned}
\tilde{\rho}_e^p(E(\delta_e))^2
\left( \frac{1}{E_D}-\frac{1}{E_0} \right)\bb{u}^\top \hat{\bb{K}}_e\bb{u} &=
-\lambda_e^- +\lambda_e^+ +\lambda \frac{\partial
g(\bb{\rho},\bb{\delta})}{\partial \delta_e}\\
&\qquad\qquad\qquad e=1,\ldots,n\\
\bb{K}\bb{u}&=\bb{f} \\
g(\bb{\rho},\bb{\delta})&\leq 0\\
-\bb{\delta}&\leq \bb{0} \\
\bb{\delta}-1&\leq \bb{0} \\
\bb{\lambda}^- &\geq \bb{0} \\
\bb{\lambda}^+ &\geq \bb{0} \\
\lambda &\geq \bb{0} \\
-\lambda_e^-\delta_e &= 0\qquad e=1,\ldots,n\\
\lambda_e^+(\delta_e-1) &= 0\qquad e=1,\ldots,n\\
\lambda \left( g(\bb{\rho},\bb{\delta}) \right) &=0
\end{aligned}
\end{equation}
However it turns out that the resulting problem is difficult to solve in practice mainly due to the presence of the complementarity constraints on the lower level. Thus, we try to get rid of those in the sequel. First of all we find
\begin{equation*}
\begin{aligned}
\frac{\partial}{\partial\delta_e}&\left(2\bb{f}^\top
	\bb{u}- \sum\limits_{e=1}^n\bb{u}^\top
	\bb{K}_e(\bb{\rho},\bb{\delta}) \bb{u}\right)\\
	&=(E(\delta_e))^2
\left( \frac{1}{E_D}-\frac{1}{E_0} \right)\bb{u}^\top \hat{\bb{K}}_e\bb{u}>0.
\end{aligned}
\end{equation*}
Using this it is seen from \eqref{eq:problem-bilevel1} that the constraint for the total uncertainty will be always at the bound and can hence
simply replaced by an equality constraint without changing the solution of the problem. This
leaves us with the bound constraints on $\bb{\delta}$, which are simply approximated using appropriate barrier terms in the objective. We obtain the problem
\begin{equation}
\begin{aligned}\label{eq:problem-bilevel}
&\min_{\bb{\rho}\in
U_{\text{ad}}} \mathcal{J}_0(\bb{\rho},\bb{\delta},\bb{u})\\
&\text{s.t.: } ({\bb{\delta},\, \bb{u}}) \in
\underset{{\bb{\delta},\, \bb{u}\in V}}{\operatorname{argmax}}\ 2\bb{f}^\top
\bb{u}-\sum_{e=1}^n\bb{u}^\top
\bb{K}_e(\bb{\rho},\bb{\delta})
\bb{u}\\
&\qquad\qquad\qquad\qquad+\mu^*\sum_{e=1}^n\left(\operatorname{log}({\delta}_e) +
\operatorname{log}(1-{\delta}_e)\right),\\
&\hspace{22mm}g(\bb{\rho},\bb{\delta})= 0,
\end{aligned}
\end{equation}
where $\mu^*$ is a small positive barrier parameter. Even though the application
of this barrier method does indeed change the solution of the optimization
problem, the influence will be minor if $\mu^*$ is chosen small enough.
Moreover, the additional term has another benefit. Using similar arguments as in the proof of \thref{thm:concavity-J-eps}, we can show for the regularized function
\begin{align}\label{eq:J-mu}
&\mathcal{J}_{\mu^*}(\bb{\rho},\bb{\delta},\bb{u}) :=\nonumber \\
&2\bb{f}^\top
\bb{u}-\sum_{e=1}^n\bb{u}^\top \bb{K}_e(\bb{\rho},\bb{\delta})
\bb{u}+\mu^*\sum_{e=1}^n\left(\operatorname{log}({\delta}_e) + \operatorname{log}(1-{\delta}_e)\right)
\end{align}
that
\begin{equation}\label{eq:hessian-pos-def-mu}
\begin{aligned}
\begin{pmatrix}\bb{x}\\ \bb{y}\end{pmatrix}^\top & \frac{\partial^2}{
(\partial ({\bb{\delta},\, \bb{u}}))^2}
\mathcal{J}_{\mu^*}(\bb{\rho},\bb{\delta},\bb{u}) \begin{pmatrix}\bb{x}\\
\bb{y}\end{pmatrix} =\\
&=\begin{pmatrix}\bb{x}\\ \bb{y}\end{pmatrix}^\top
\frac{\partial^2}{ (\partial ({\bb{\delta},\, \bb{u}}))^2}
\bigg[\mathcal{J}(\bb{\rho},\bb{\delta},\bb{u})\\
&\qquad+\mu^*\sum_{e=1}^n\Big(\operatorname{log}({\delta}_e)+ \operatorname{log}(1-{\delta}_e)\Big)\bigg] 
\begin{pmatrix}\bb{x}\hspace{1cm}\\
\bb{y}\end{pmatrix}\\ 
&\hspace{-.6em} = \begin{cases}
\leq \mu^* \sum_{e=1}^n x_e^2\left(
\frac{1}{\delta_e^2}+\frac{1}{(1-\delta_e)^2} \right) \ > 0 & \bb{x}\neq \bb{0}\\
= -\bb{y}^\top \bb{K} \bb{y} < 0 & \bb{x} = \bb{0}
\end{cases}
\end{aligned}
\end{equation}
for all $\left(\bb{x}^\top, \bb{y}^\top\right)\neq \bb{0}$ and thus strict
concavity of $\mathcal{J}_{\mu^*}$ in the joint variable $(\bb{\delta},\bb{u})$
for every $\mu^*>0$. \\

This implies that the constraint in \eqref{eq:problem-bilevel} can be equivalently stated through optimality conditions, which read as
{
\medmuskip=2.5mu
\thinmuskip=3.5mu
\thickmuskip=4mu
\begin{equation}\label{eq:kkt-mu}
\begin{aligned}
\tilde{\rho}_e^p (E(\delta_e))^2
&\left( \frac{1}{E_D}-\frac{1}{E_0} \right)\bb{u}^\top
\hat{\bb{K}}_e\bb{u}+\\
\mu^*\left( \frac{1}{\delta_e} - \frac{1}{1-\delta_e}
\right)
&= -\lambda\frac{\partial g}{\partial \delta_e}, \quad \text{for } e=1,\ldots,n\\
\bb{K}\bb{u}&=\bb{f} \\
g(\bb{\rho},\bb{\delta}) &= 0\\
\end{aligned}
\end{equation}
}
Now, having eliminated all complementarity constraints, we may formulate a new bilevel optimization problem as
{
\medmuskip=2.5mu
\thinmuskip=3.5mu
\thickmuskip=4mu
\begin{equation}
\begin{aligned}\label{eq:problem-nested}
&\min_{\bb{\rho}\in
	U_{\text{ad}},\bb{\delta}\in \RR^n, \bb{u} \in \RR^m}\mathcal{J}_0(\bb{\rho},\bb{\delta},\bb{u})\\
&\quad \text{s.t.: \eqref{eq:kkt-mu} holds.}
\end{aligned}
\end{equation}
}
and its implicit counterpart as
\begin{equation}
\begin{aligned}\label{eq:problem-nested2}
&\min_{\bb{\rho}\in
U_{\text{ad}}}\mathcal{J}_0(\bb{\rho},\bb{\delta}(\bb{\rho}),\bb{u}(\bb{\rho})).
\end{aligned}
\end{equation}
Here $\bb{\delta}(\bb{\rho})$ and $\bb{u}(\bb{\rho})$ are the solution operators associated with \eqref{eq:kkt-mu}. Note that $\bb{\delta}(\bb{\rho})$ and $\bb{u}(\bb{\rho})$ are well defined as the system \eqref{eq:kkt-mu} has a unique solution for each feasible $\bb{\rho}$ thanks to the strict concavity of \eqref{eq:J-mu}.

For the latter problem, we can now use a standard adjoint approach for the
calculation of the gradient. In order to do so, we use the Lagrangian
corresponding to the lower-level problem
\begin{equation}\label{L-mu*}
\begin{aligned}
\mathcal{L}_{\mu^*}(\bb{\rho},\bb{\delta},\bb{u}, \lambda) := &2\bb{f}^\top
\bb{u}-\sum_{e=1}^n\bb{u}^\top
\bb{K}_e(\bb{\rho},\bb{\delta})
\bb{u}\\
&+\mu^*\sum_{e=1}^n\left(\operatorname{log}({\delta}_e)+\operatorname{log}(1-{\delta}_e)\right)-\lambda g(\bb{\rho},\bb{\delta}),
\end{aligned}
\end{equation}
in order to rewrite the KKT conditions \eqref{eq:kkt-mu} as
\begin{equation} \label{KKT_lagr}
\frac{\partial \mathcal{L}}{\partial (\bb{\delta},\bb{u}, \lambda)}=\bb{0}.
\end{equation}
Now, using the adjoint calculus for problem (\ref{eq:problem-nested2}), we get:
\begin{equation}\label{Lagrangian}
\frac{\partial
\mathcal{J}_0(\bb{\rho},\bb{\delta}(\bb{\rho}),\bb{u}(\bb{\rho}))}{\partial
\rho_e}=\frac{\partial \mathcal{J}_0}{\partial\rho_e}
+\frac{\partial^2\mathcal{L}_{\mu^*}}{\partial\rho_e\partial\bb{\delta}}\bb{\tau}_{\bb{\delta}}
+\frac{\partial^2\mathcal{L}_{\mu^*}}{\partial\rho_e\partial\bb{u}}\bb{\tau}_{\bb{u}}
+\frac{\partial^2\mathcal{L}_{\mu^*}}{\partial\rho_e\partial\lambda}\tau_\lambda,
\end{equation}
where $\left(\bb{\tau}_{\bb{\delta}},\bb{\tau}_{\bb{u}},\tau_\lambda \right) $
are the solutions of the adjoint equations
\begin{equation}\label{eq:adjoint-smoothed}
\begin{pmatrix}
\frac{\partial^2\mathcal{L}_{\mu^*}}{(\partial\bb{\delta})^2} &
\frac{\partial^2\mathcal{L}_{\mu^*}}{\partial\bb{\delta}\partial\bb{u}} & 
-\frac{\partial g}{\partial \bb{\delta}}\\
\frac{\partial^2\mathcal{L}_{\mu^*}}{\partial\bb{u}\partial\bb{\delta}} &
-\bb{K} & \bb{0} \\
-\frac{\partial g}{\partial \bb{\delta}} & \bb{0} & \bb{0}
\end{pmatrix}
\begin{pmatrix}
\bb{\tau}_{\bb{\delta}} \\ \bb{\tau}_{\bb{u}} \\ \tau_\lambda
\end{pmatrix} = \begin{pmatrix}
\frac{\partial \mathcal{J}_0}{\partial\bb{\delta}} \\
\frac{\partial \mathcal{J}_0}{\partial\bb{u}} \\
\frac{\partial \mathcal{J}_0}{\partial\lambda}
\end{pmatrix}.
\end{equation}
Finally, if we slightly modify problem \eqref{eq:problem-nested2} by replacing the upper level objective $J_0$ by $J_{\mu^*}$ and use again formula \eqref{L-mu*} and \eqref{KKT_lagr}, the right hand side simplifies to
\begin{equation*}
\begin{aligned}
&\begin{pmatrix}
\frac{\partial\mathcal{J}_{\mu^*}}{\partial\bb{\delta}} \\
\frac{\partial\mathcal{J}_{\mu^*}}{\partial\bb{u}} \\
\frac{\partial\mathcal{J}_{\mu^*}}{\partial\lambda}
\end{pmatrix} = \begin{pmatrix}
\frac{\partial\mathcal{L}_{\mu^*}}{\partial\bb{\delta}}+\lambda\frac{\partial
g}{\partial \bb{\delta}}
\\
\bb{f}-\bb{K}\bb{u} \\ \bb{0}
\end{pmatrix} = \begin{pmatrix}
\lambda\frac{\partial g}{\partial \bb{\delta}} \\ \bb{0} \\ \bb{0}
\end{pmatrix}.
\end{aligned}
\end{equation*}
With this modified right hand side, we can directly find
$\left(\bb{\tau}_{\bb{\delta}},\bb{\tau}_{\bb{u}},\tau_\lambda \right)=\left(
\bb{0},\bb{0},-\lambda \right)$ to be a solution of system \eqref{eq:adjoint-smoothed}. Thus, no adjoint problem has to be solved and we obtain the derivative formula: 
\begin{equation}\label{eq:derivative-minimax}
\begin{aligned}
\frac{\partial
\mathcal{J}_{\mu^*}(\bb{\rho},\bb{\delta}(\bb{\rho}),\bb{u}(\bb{\rho}))}{\partial
\rho_e}&=\frac{\partial\mathcal{J}_{\mu^*}(\bb{\rho},\bb{\delta},\bb{u})}{\partial
\rho_e}-\lambda \frac{\partial g}{\partial \rho_e}\\
&=-\bb{u}^\top\frac{\partial \bb{K}}{\partial\rho_e}\bb{u}-\lambda
\frac{\partial g}{\partial \rho_e}.
\end{aligned}
\end{equation}
It is seen that this formula has precisely the same structure as the derivative formula for $F_\epsilon$ in \secref{sec:minimax}. This is not too surprising, as, by changing the objective as above, \eqref{eq:problem-nested2} is essentially turned into a minimax problem and thus the generic theory for the minimax-problems in \secref{sec:minimax} also applies to this situation.

\section{Algorithmic solution}
\subsection{Lower-level problem solution}
Even though we have seen that gradients of the upper-level problem can be
easily computed using, e.~g., the minimax theorems in \secref{sec:minimax}, the lower-level problem still has to
be solved very precisely in each outer iteration for a sufficiently accurate
gradient. Using, for instance, the standard topology optimization algorithm MMA for the solution of the inner problem will
hence lead to an very high number of total iterations required. On the other hand, the
sparse Hessian of the lower-level maximization problem and its strict concavity
lend themselves to the use of second-order optimization methods and their much
better speed of convergence. For this, we take a quick look at interior-point
methods.

Interior-point methods replace inequalities with exactly the same barrier term
that was used in \secref{sec:smoothed} and iteratively solve the resulting problems for a
decreasing sequence of barrier parameters $\mu\searrow 0$. For the sake of completeness, We demonstrate this
concept for a simple model problem with objective $f:\RR^n\to\RR$, and a constraint
$g:\RR^n\to\RR$: The original problem
\begin{align*}
\min_{\bb{x}\in\RR^n}&f(\bb{x})\\
\text{ s.\,t.: }&g(\bb{x})=0\\
&-\bb{x}\leq\bb{0}
\end{align*}
is converted into a sequence of barrier problems parametrized by the barrier parameter $\mu>0$:
\begin{align*}
\min_{\bb{x}\in\RR^n}&f-\mu\sum_{e=1}^n\log(x_e) \\
\text{ s.\,t.: }&g(\bb{x})=0\\
\end{align*}
Each of the latter problems can be now approached by solving the optimality conditions
\begin{align*}
\nabla f(\bb{x})+\lambda\nabla g(\bb{x}) -\left(\frac{\mu}{x_e}\right)_{e=1}^n
&= \bb{0},\\
g(\bb{x})&=0.
\end{align*}
In order to get a full picture of the error made in comparison to the KKT system
of the original optimization problem, the term $(\mu/x_e)_{e=1}^n\geq\bb{0}$ is
interpreted as a Lagrange vector $\bb{z}\geq\bb{0}$ of the inequality
constraints. Substituting $\bb{z}:=(\mu/x_e)_{e=1}^n$ in the above
optimality conditions then gives the primal-dual equations
\begin{align*}
\nabla f(\bb{x})+\lambda\nabla g(\bb{x}) -\bb{z}
&= \bb{0},\\
g(\bb{x})&=0,\\
x_e z_e - \mu &= 0,\qquad e=1,\ldots,n.
\end{align*}
Note that for $\mu=0$ these are the KKT conditions of the original optimization
problem including the complementarity constraints. Hence, the last line in the
equation set is sometimes called perturbed complementarity and $\mu>0$ is the
desired complementarity error. This system, sometimes also referred to as central path equations is now solved repeatedly for decreasing $\mu$.\\
For the problem considered in this work, the barrier objective obtained by the interior point procedure is exactly
the cost function $\mathcal{J}_{\mu}$, see \eqref{eq:J-mu}. Hence, we can use existing interior point
solvers and simply stop the decreasing sequence of $\mu$ at $\mu^*$. Moreover,
this automatically implies a numerical advantage of the regularization
with a barrier term compared with the $\epsilon$-regularization of
\secref{sec:minimax}: The solution of the lower-level problem is faster because
the decreasing $\mu$ sequence can be stopped consistently sooner. Furthermore, we
have seen that $\mu$ has a clear interpretation as complementarity error, while
the influence of $\epsilon$ on optimization results is much less clear.
\\
\subsection{Algorithm}
We combine two different optimization algorithms for the upper- and lower-level
problems. We solve the lower-level optimization problem in
$(\bb{\delta},\bb{u})$ with the second order interior point optimizer IPOPT, see \cite{wachter2006implementation}, and then compute the gradients w.\,r.\,t.
$\bb{\rho}$ using formula (\ref{eq:derivative-minimax}). For the upper-level problem in
$\bb{\rho}$, we use the Method of Moving Asymptotes (MMA, see
\cite{Svanberg:02}). This procedure is repeated until convergence is achieved.
\begin{algorithm}[H]    
        \caption{Pseudoalgorithm for topology optimization with worst-case
        material uncertainties}
        \label{alg:robust}
                set initial iterates $\bb{\rho}^{0},\,\bb{\delta}^0$,
                number of iterations \textit{maxIter}, volume bound $V$, total uncertainty bound $D$,
                filter radius $R$ and smoothing parameter $\mu^*$\;
                solve standard topology optimization problem (e.\,g. using MMA)
                $\longrightarrow
                \bb{\rho}^1,\,\bb{u}^0(\bb{\rho^1})$\label{line:simp-solution}\; 
                \For{$k=1,...,$maxIter}{
                    solve lower-level problem using
                    IPOPT $\longrightarrow
                    \bb{\delta}^k(\bb{\rho}^k),\,\bb{u}^k(\bb{\rho}^k)$
                    \label{line:lower-level-problem}\; compute sensitivities of
                    outer cost functional using formula
                    (\ref{eq:derivative-minimax})\; do single MMA step
                    $\longrightarrow \bb{\rho}^{k+1}$\;
                \If{convergence criterium is satisfied}{
                 	break\;}
                 }
                   
                postprocess data and visualize results\;
\end{algorithm}
Note that line \ref{line:simp-solution} is optional, however doing at least a couple of
``cheaper'' non-robust optimization steps will result in being closer to an
optimal topology as a starting value and the convergence speed of the
subsequent robust iterations can be improved. For the results in this work, the (pure)
topology optimization is done until convergence, as it is used to obtain a
reference compliance value in Section 5.\\
\subsubsection{Software settings}
As a default, IPOPT starts with a smoothing parameter of \texttt{mu\_init}=0.1,
which is subsequently decreased during the optimization process unto a final
value of
\texttt{mu\_target}. The final smoothing parameter \texttt{mu\_target}
corresponds to $\mu^*$ in our formulation and we set it to
\texttt{mu\_target}=$10^{-7}$. Note that the choice of $\mu^*$ depends on the
scaling of the inner cost function. If it is chosen too large compared to the
cost function, the barrier term is large enough to push $\bb{\delta}$
significantly towards intermediate values, thus noticeably influencing the
optimization results. If it is chosen too small, the intended strict concavity
and regularization effects get less pronounced. 
More importantly, the number of iterations needed for solving
the lower-level problem as well as the computation times decrease when choosing
$\mu^*$ larger, as this determines how early IPOPT's usual decrease of $\mu$ to
0 may be stopped. As the initial barrier parameter, in the first
iteration we use the default value \texttt{mu\_init}=0.1.
After the first iteration, when we can reuse the previous iterates of
$\bb{\delta}$, $\bb{u}$ and IPOPT's Lagrangian parameters as a starting value,
we start IPOPT directly with \texttt{mu\_init}=$\mu^*$.\\
As termination criterion, IPOPT uses essentially the maximum norm of the KKT
error for the central path equations associated with \texttt{mu\_target}. The termination
criterion is set to a very strict value of \texttt{tol}=$10^{-10}$. Furthermore,
the terminal maximum constraint violation and complementarity error can be set
seperately as well. We use the values of
\mbox{\texttt{constr\_viol\_tol}=$10^{-10}$}
and \texttt{compl\_inf\_tol}=$0.0001$, respectively. The strict termination
criteria ensure that the optimality assumptions for the computation of the
upper-level gradients are fulfilled sufficiently accurate. \\
As the linear system solver within IPOPT, we use the HSL Mathematical Software
Library \cite{hsl2013collection}.

\section{Numerical results}
\subsection{Numerical example setup}
As exemplary setup we use a loaded cantilever fixed on the left side with a
rectangular design domain of the dimension $[0,2]\times[0,1]$ as
depicted in \figref{fig:cantilever}. The cantilever is loaded with a line load
at the lower right bottom starting at $x=1.9$ of total magnitude $0.3$ in
direction $-y$. The domain is discretized using $300\times 150$ uniform
rectangular finite elements.
We use a linearly decaying density filter with a radius of $0.045$ corresponding
to the length of $4.5$ elements. As resource constraints, we use a total
relative material volume of 50\% ($V=0.5$). The penalty parameter is set to
$p=4$ and the lower bound of the topology variable to $\rho_{\text{min}}=0.01$.
The maximum number of iterations is $500$.\\
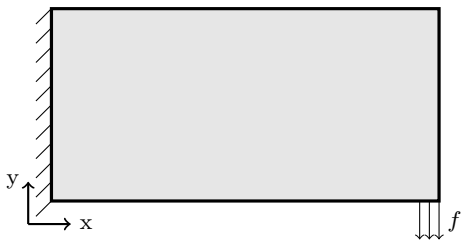
\begin{figure}[!hb] \centering

  \begin{tikzpicture}[scale=2.55]
  \draw[very thick,fill=black!10] (0,0) rectangle (2.,1);
  \foreach \y in {.0,.1,...,.91,1.0}
  { \draw (0,\y) -- (-.08, \y - .08); }
  \draw[->] (1.9,0.0) -- (1.9, -0.2) node[right,anchor=south west]{};
  \draw[->] (1.95,0.0) -- (1.95, -0.2) node[right,anchor=south west]{};
  \draw[->] (2.,0.0) -- (2., -0.2) node[right,anchor=south west] {$f$};
  \draw[->,thick] (-0.12,-0.12) -- (0.1,-0.12) node[right] {x} coordinate(x axis);
  \draw[->,thick] (-0.12,-0.12) -- (-0.12,0.1) node[left] {y} coordinate(y
  axis);
 \end{tikzpicture}
 \label{fig:example1}
 \caption{Loaded cantilever example}\label{fig:cantilever}
\end{figure}
The starting value for the robust optimization is for all results the solution
of the respective SIMP topology optimization problem without material
uncertainties, i.e. ($\bb{\delta}\equiv 0$). This reference problem and the resulting topology will be denoted as \textit{nominal problem} and \textit{nominal topology}, respectively. For
this nominal topology, we compute the compliance using the
material with Young's modulus $E_0$ as a \textit{nominal compliance}. This
nominal ``best-case compliance'' is then compared to the worst-case
compliance of the nominal topology, as well as the best- and worst-case
compliances of the robustly optimized topology.
\subsection{Additive manufacturing}
The complex process in additive manufacturing
can lead to an inhomogeneous microstructure of the material, e.g. by microscopic
pores.
This introduces variances in the local effective material properties. This
variance will be modeled as a uncertainty in the Young's modulus, which can also
be observed in measurements. The measured maximum and minimum Young's modulus
will then be used in the worst-case framework detailed in this work.
Specifically, the uncertain material properties are allowed to deviate by 30\%,
i.\,e. the normalized values in the optimization problem are $E_D = 0.7$ and
$E_0= 1$. These values are based on tensile testing results for AlSi10Mg
\citep{vdihandbuch}. Note that the variations in Young's modulus will change for
different processes, materials, machines and process parameters. E.\,g. for
reused polyamide powder, the elastic modulus may decrease rapidly up to half the
original value after several uses \citep{wegner2012betrachtung} while for the
maraging steel 18Ni-300 the variations may be closer to 14\% (see
\cite{hermann2016achievable}). As the amount of material variations occurring is
even harder to estimate, we use a number of different volumes for the
degeneration variable ($D=0.03,\,
D=0,06$ and $D=0.10$). However, we restrict ourselves in the visualization to a single representative. \\
Before stating the results, we can estimate the maximum and minimum effect that
these numbers can theoretically have on the compliance. As a change in material has a linear
influence on the compliance, a decrease in material stiffness from $1$ to $0.7$
on the whole domain would amount to $1/E_D=1/0.7\approx 143\%$ times the
compliance compared to using only the stiffest material. Similarly, we can
distribute the fraction of degenerated material uniformly on all elements to get
a lower bound of $((1-D)/E0+D/E_D)$ for $\bb{\delta}\equiv D$. However, as we can assume an
approximate 0-1-distribution for the topology, a better estimate would be to
distribute the degeneration variable only on the elements with $\bb{\rho}=1$.
Neglecting that the topology variable is not actually binary, this gives us
$\bb{\delta}\equiv D/V$ and a (slightly too large) lower bound estimate of the
minimum effect of the degradation of $((1-D/V)/E0+(D/V)/E_D)$. For the
total uncertainty bounds of $D=0.03,\, D=0,06$ and $D=0.1$, this gives
a compliance increase of about $2.57\%, 5.14\%$ and $8.57\%$, respectively.

The compliance values obtained from numerical results can be found in
\tabref{tab:am-linear} and are denoted as the increase w.\,r.\,t.
the compliance of the nominal topology (standard SIMP result) evaluated with material $E_0$ only. As expected, the worst-case compliance values greatly surpass
those of the lower bound estimates.
Furthermore, the compliance of the robust topology evaluated with $\bb{\delta}\equiv 0$ slightly
deteriorates, while the worst-case compliance improves about twice as
much, if we pass from the nominal to the robustly optimized topology. Looking at the visualization of some results in \figref{fig:am-linear}, we
first observe that the robust optimization does not change the nominal SIMP
topology a lot. Only on the left side we see that the structure gets slightly thicker
along the support. It also becomes apparent that the worst case distribution of the
degradation variable $\bb{\delta}$ is almost binary. \\
\begin{table}[!ht]
\caption{From left to right, the total uncertainty bound, the compliance of the worst-case topology with homogeneous material $E_0$ ($\bb{\delta}\equiv 0$), the compliance of the nominal topology with worst-case material distribution and the compliance of the worst-case topology with worst-case material distribution are displayed. All compliance values are specified as increase in comparison to the compliance of the nominal topology with nominal material, i.e. $\bb{\delta}\equiv 0$ on $\Omega$. The chosen material parameters are $E_0=1$ and $E_D=0.7$ in all cases.}
\label{tab:am-linear}
{\setlength{\tabcolsep}{.61em}
\begin{tabular}{|c|c|c|c|}
\hline
\makecell{total\\ uncertainty}  & \makecell{w.-c. topo,\\ $\bb{\delta}\equiv 0$} & \makecell{nom. topo,\\ worst-case $\bb{\delta}$}  &  \makecell{w.-c. topo,\\ worst-case $\bb{\delta}$} \\
\hline
$D = 0.03$ & +0.182\% & +7.69\% & +7.25\% \\
\hline
$D = 0.06$ & +0.317\% & +11.5\% & +10.8\% \\
\hline
$D = 0.10$ & +0.383\% & +15.6\% & +14.8\% \\
\hline
\end{tabular}
}
\end{table}


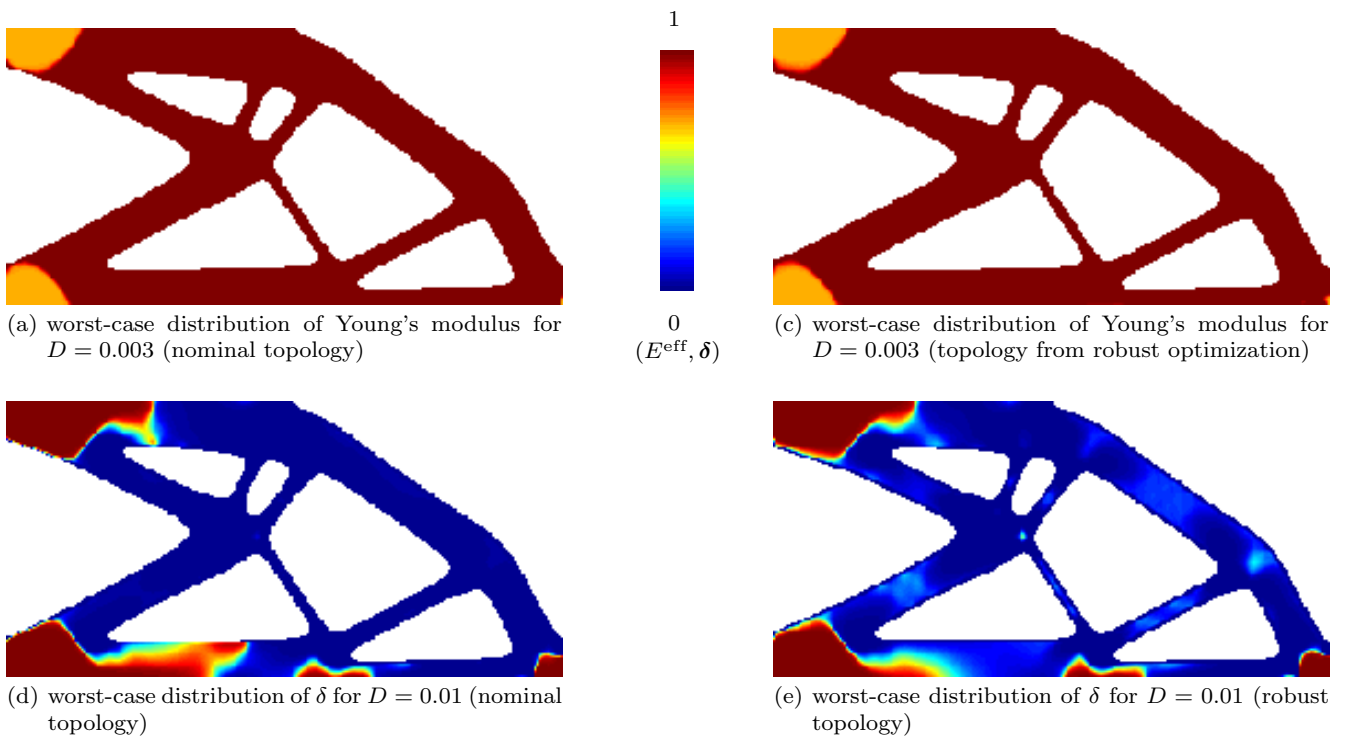
\begin{figure*}[!htb]
\centering
\subfigure[worst-case distribution of Young's modulus for $D=0.003$ (nominal topology)]
{\includegraphics[width=.42\textwidth]{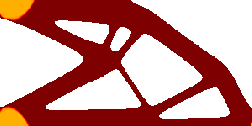}\label{amlinear/3e-2/Es}}\hfill
\raisebox{-9mm}{\subfigure{ \begin{tikzpicture}[scale=1.3]
  \node at (0,0)
  {\includegraphics[trim=0 0 0 0, clip,width=.025\textwidth]{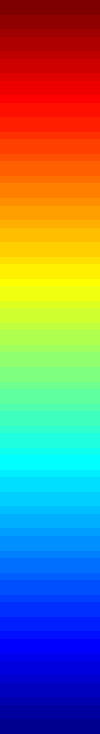}};
  \node[text width=.025\textwidth] at (0.08,1.55) {$1$};
  \node[text width=.025\textwidth] at (0.08,-1.55) {$0$};
  \node[text width=.025\textwidth] at (-0.26,-1.85) {$(E^{\text{eff}}, \bb{\delta}$)};
\end{tikzpicture}
}}\hfill
 \subfigure[worst-case distribution of Young's modulus for $D=0.003$ (topology from robust
 optimization)]
 {\includegraphics[width=.42\textwidth]{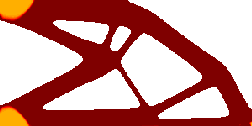}\label{amlinear/3e-2/Ef}}
\subfigure[worst-case distribution of $\delta$ for $D=0.01$ (nominal topology)]
 {\includegraphics[width=.42\textwidth]{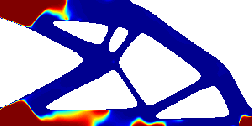}\label{amlinear/1e-1/ds}}\hfill
\subfigure[worst-case distribution of $\delta$ for $D=0.01$ (robust topology)]
 {\includegraphics[width=.42\textwidth]{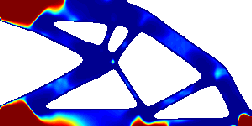}\label{amlinear/1e-1/df}}
 \caption{Optimization results showing worst-case material variations for
 different total uncertainty bounds. Material parameters are chosen as $E_0=1$, $E_D=0.7$. Left column: worst-case material distributions for nominal topology. Right column: worst-case material distributions for robust topology.}\label{fig:am-linear}
\end{figure*}

Next we test the impact of changing the uncertainty set. The motivation for this is that
a worst-case behaviour in the sense that the material properties are systematically weakened is not really expected in additive manufacturing practice. Looking e.\,g. at \cite{vdihandbuch}, we rather expect the Young's
modulus to fluctuate around a mean value. 
An uncertainty model reflecting this behaviour has already been defined in Section 2 and is repeated here for the convenience of the reader:
\begin{equation*}
\begin{aligned}
  \Delta_{\text{ad}}^{\text{AM}}(\bb{\rho}) = \bigg\{ \bb{\delta}\in[0,1]^n :\ 
  &e=1,\ldots,n,\ \sum_{e=1}^n
  \frac{v_e\bb{\tilde{\rho}}_e^p\delta_e}{|\Omega|} = 0.4,\\
  &\sum_{e=1}^n \frac{v_e(\bb{\tilde{\rho}}_e^p\delta_e-0.4)^2}{|\Omega|} \leq D \bigg\}.
\end{aligned}
\end{equation*}
This choice of $\Delta_{\text{ad}}$ causes that the material defects to fluctuate around an
expected value between the stiffest and the least stiff material.
Note that by the constraint $(v_e\bb{\tilde{\rho}}_e^p\delta_e)/|\Omega| = 0.4$, for $E_0=1$ and $E_D=0.7$ we get an average Young's modulus of $(0.6/E_0 + 0.4/E_D)^{-1}\approx 0.8537$, which is very close to the true mean value of $0.85$.
Furthermore, the quadratic constraint $\sum_{e=1}^n (v_e(\bb{\tilde{\rho}}_e^p\delta_e-0.4)^2)/|\Omega| \leq D$ models that values of $\bb{\delta}$ close to the average of $0.4$ (and accordingly Young's moduli close to the expected value) are comparatively cheap, while values further away are significantly more expensive. 

For this setting, we use the compliance of the nominal topology
with the average material generated by $\bb{\delta} \equiv 0.4$ on the
whole domain as reference compliance corresponding to the mean stiffness expected
from tensile testing.
The worst-case compliance is expected to be much closer to the reference
value now, as the allowed variation from the reference is halved. In addition the choice of the uncertainty set implies that weak material at one location has to be compensated in equal amounts by stiffer material in other places. We consider different bounds on the total material variations defined now by the quadratic constraint. Specifically we use $D\in\{ 0.01, 0.02, 0.03, 0.04, 0.06, 0.1 \}$. The relative compliance values obtained by numerical experiments for this setting can be found in \tabref{tab:am-average}; selected results are
visualized in Figure 
\ref{fig:am-average2}. We note that
the compliance values are now rather close to each other. For $\bb{\delta} \equiv 0.4$ the
robust topology is only slightly worse than the nominal topology while still
consistently outperforming the latter by about twice that margin for the worst-case
material distribution. Looking at the images, the worst-case
material distribution looks a lot more consistent with the continuous variations of the
Young's modulus that we see in tensile testing.
While the worst-case material distribution still shows material with the lowest
possible stiffness in the most critical areas of the cantilever, now material
stiffnesses in the whole allowed range can be found, especially for the
lower total uncertainty bounds, e.g. $D=0.003$. 

In summary, the presented results suggest that in additive manufacturing the
influence of the uncertain material parameters on the compliance of topology
optimized parts is only minor. Nevertheless, the robust optimization methods are
consistently decreasing the influence of the worst-case material distribution.
\begin{table}[!bt]
\caption{Compliance increase in comparison to nominal topology with
``average material'' ($\bb{\delta}\equiv 0.4$); see \tabref{tab:am-linear} for a detailed explanation of columns 2 to 4. Material parameters are again chosen as $E_0=1$ and $E_D=0.7$.}
\label{tab:am-average}
{\setlength{\tabcolsep}{.61em}
\begin{tabular}{|c|c|c|c|}
\hline
\makecell{total\\ uncertainty}  & \makecell{w.-c. topo,\\ worst-case $\bb{\delta}$} & \makecell{w.-c. topo,\\ $\bb{\delta}\equiv 0.4$} & \makecell{nom. topo,\\ worst-case $\bb{\delta}$}\\ \hline
$D = 0.01$ & +0.050\% & +3.07\% & +2.93\% \\
\hline
$D = 0.02$ & +0.108\% & +3.86\% & +3.63\% \\
\hline
$D = 0.03$ & +0.170\% & +4.32\% & +3.98\% \\
\hline
$D = 0.04$ & +0.195\% & +4.58\% & +4.18\% \\
\hline
$D = 0.06$ & +0.304\% & +4.90\% & +4.28\% \\
\hline
$D = 0.10$ & +0.331\% & +5.08\% & +4.29\% \\
\hline
\end{tabular}
}
\end{table}

\begin{figure*}[!ht]
\centering
\subfigure[worst-case distribution of $\delta$ for $D=0.003$ (nominal topology)]
 {\includegraphics[width=.42\textwidth]{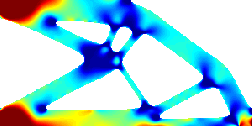}\label{amavg/3e-2/ds}}\hfill
\raisebox{-9mm}{\subfigure{ \begin{tikzpicture}[scale=1.3]
  \node at (0,0)
  {\includegraphics[width=.025\textwidth]{colormap-E.png}};
  \node[text width=.025\textwidth] at (0.08,1.55) {$1$};
  \node[text width=.025\textwidth] at (0.08,-1.55) {$0$};
  \node[text width=.025\textwidth] at (-0.26,-1.85) {$(E^{\text{eff}}, \bb{\delta}$)};
\end{tikzpicture}
}}\hfill
\subfigure[worst-case distribution of $\delta$ for $D=0.003$ (robust topology)]
 {\includegraphics[width=.42\textwidth]{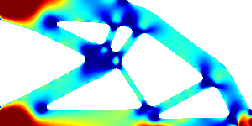}\label{amavg/3e-2/df}}

\subfigure[worst-case distribution of $\delta$ for $D=0.006$ (nominal topology)]
 {\includegraphics[width=.42\textwidth]{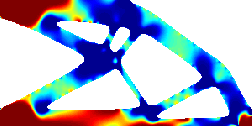}\label{amavg/6e-2/ds}}\hfill
\subfigure[worst-case distribution of $\delta$ for $D=0.006$ (robust topology)]
 {\includegraphics[width=.42\textwidth]{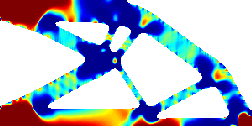}\label{amavg/6e-2/df}}
 \subfigure[worst-case distribution of Young's modulus for $D=0.006$ (nominal SIMP topology)]
{\includegraphics[width=.42\textwidth]{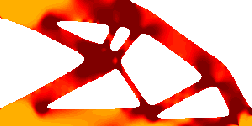}\label{amavg/6e-2/Es}}\hfill
 \subfigure[worst-case distribution of Young's modulus for $D=0.006$ (topology from robust
 optimization)]
 {\includegraphics[width=.42\textwidth]{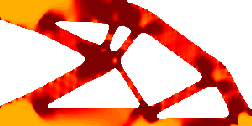}\label{amavg/6e-2/Ef}}
 \caption{Optimization results obtained using the uncertainty model $\Delta_{\text{ad}}^{\text{AM}}(\bb{\rho})$. Displayed are worst-case material variations for
 different total uncertainty bounds and material parameters $E_0=1$, $E_D=0.7$. Left column: worst-case material distributions for nominal topology. Right column: worst-case material distributions for robust topology.}\label{fig:am-average2}
\end{figure*}
\subsection{Robust topology optimization with material degeneration}
As for the additive manufacturing scenarios, only a minor effect of material uncertainties was observed, a second scenario is studied, which allows more easily to investigate the strength of our robust model and algorithm. Specifically, we will use a significantly lower
material parameter $E_D$ modelling that the material is heavily degraded at local spots. As this situation clearly corresponds to a degradation scenario, we apply the uncertainty model \eqref{eq:delta-ad-rho} again.
We start our investigations with values $E_0=1$ and $E_D=0.01$. For these
we consider a large number of different total uncertainty bounds ranging from
$D=0.0002$ to $D=0.008$, as the robust topology and compliances are heavily
influenced by this bound. \\
\begin{table}[!b]
\caption{Compliance increases in comparison to nominal topology with fully
stiff material $E_0$ ($\bb{\delta}\equiv 0$); for a detailed explanation of columns 2 to 4, see \tabref{tab:am-linear}. Material parameters are now chosen as $E_0=1$ and $E_D=0.01$ corresponding to a serious degradation scenario.}
\label{tab:rto}
{\setlength{\tabcolsep}{.61em}
\begin{tabular}{|c|c|c|c|}
\hline
\makecell{total\\ uncertainty}  & \makecell{w.-c. topo,\\ worst-case $\bb{\delta}$} & \makecell{w.-c. topo,\\ $\bb{\delta}\equiv 0$} & \makecell{nom. topo,\\ worst-case $\bb{\delta}$}\\
\hline
$D = 0.0002$ & +1.19\% & +17.4\% & +13.6\% \\
\hline
$D = 0.0003$ & +1.96\% & +25.4\% & +18.5\% \\
\hline
$D = 0.0004$ & +2.80\% & +34.3\% & +22.7\% \\
\hline
$D = 0.0005$ & +3.38\% & +40.9\% & +26.3\% \\
\hline
$D = 0.0006$ & +3.87\% & +48.0\% & +29.7\% \\
\hline
$D = 0.001$ & +5.71\% & +80.1\% & +42.5\% \\
\hline
$D = 0.002$ & +7.47\% & +148\% & +71.7\% \\
\hline
$D = 0.003$ & +8.01\% & +216\% & +99.8\% \\
\hline
$D = 0.004$ & +8.67\% & +283\% & +127\% \\
\hline
$D = 0.008$ & +10.5\% & +532\% & +230\% \\
\hline
\end{tabular}
}
\end{table}
\begin{figure*}[!ht]
\centering
\subfigure[worst-case distribution of Young's modulus for $D=0.0002$ (nominal topology)]
 {\includegraphics[width=.42\textwidth]{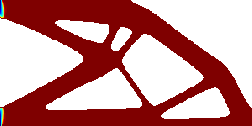}\label{rto/2e-4/ds}}\hfill
 \raisebox{-9mm}{\subfigure{ \begin{tikzpicture}[scale=1.3]
  \node at (0,0)
  {\includegraphics[width=.025\textwidth]{colormap-E.png}};
  \node[text width=.025\textwidth] at (0.08,1.55) {$1$};
  \node[text width=.025\textwidth] at (0.08,-1.55) {$0$};
  \node[text width=.025\textwidth] at (-0.16,-1.85) {$(E^{\text{eff}}$)};
\end{tikzpicture}
}}\hfill
\subfigure[worst-case distribution of Young's modulus for $D=0.0002$ (robust topology)]
 {\includegraphics[width=.42\textwidth]{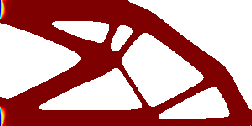}\label{rto/2e-4/df}}
\subfigure[worst-case distribution of Young's modulus for $D=0.0006$ (nominal topology)]
 {\includegraphics[width=.42\textwidth]{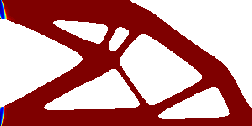}\label{rto/6e-4/ds}}\hfill
\subfigure[worst-case distribution of Young's modulus for $D=0.0006$ (robust topology)]
 {\includegraphics[width=.42\textwidth]{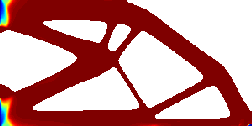}\label{rto/6e-4/df}}
\subfigure[worst-case distribution of Young's modulus for $D=0.001$ (nominal topology)]
 {\includegraphics[width=.42\textwidth]{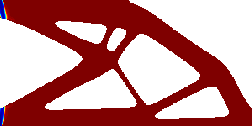}\label{rto/1e-3/ds}}\hfill
\subfigure[worst-case distribution of Young's modulus for $D=0.001$ (robust topology)]
 {\includegraphics[width=.42\textwidth]{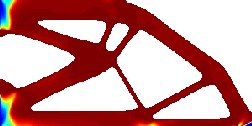}\label{rto/1e-3/df}}
\subfigure[worst-case distribution of Young's modulus for $D=0.004$ (nominal topology)]
 {\includegraphics[width=.42\textwidth]{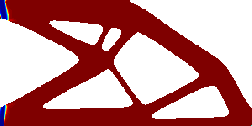}\label{rto/4e-3/ds}}\hfill
\subfigure[worst-case distribution of Young's modulus for $D=0.004$ (robust topology)]
 {\includegraphics[width=.42\textwidth]{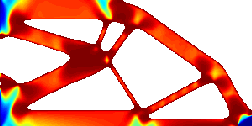}\label{rto/4e-3/df}}
 \caption{Optimization results showing worst-case material variations for
 increasing total uncertainty bounds. The material parameters $E_0=1$, $E_D=0.01$ correspond to the situation of a serious material degradation. Left column: worst-case material distributions for nominal topology. Right column: worst-case material distributions for robust topology.}\label{fig:rto2}
\end{figure*}
Selected optimal results are visualized in 
Figure \ref{fig:rto2} and the relative compliance values can be found in \tabref{tab:rto}. 
As in the case of additive manufacturing, the
cantilever gets thicker at the support, but this time by a much larger margin.
It can be seen quite nicely, how -- with increasing material uncertainty -- more and more of the support is occupied with ``degenarate'' material.
The compliance values in \tabref{tab:rto} reveal that now the worst-case compliance is decreased immensely for the robust topology compared to the nominal one, while sacrificing comparatively little structural
performance for the case without material degeneration. For the nominal SIMP
topology, degradation is almost exclusively observed at the support, whereas for the robustly optimized topology degradation starts appearing at the load already for $D=0.0002$ and is getting more pronounced for higher values of $D$.

At this point we would like to recall that -- as outlined in \secref{sec:dis_interpol} -- in the degeneration case, i.e. $E_D \ll E_0$, the difference between the inverse interpolation model \eqref{eq:matdelta} and the linear interpolation model \eqref{eq:damage-model} may become substantial. 
In order to investigate this in practice, we would first like
to get a sense of what results for the linear interpolation would look like.
Simply replacing model \eqref{eq:matdelta} by \eqref{eq:damage-model} in the lower-level problem however leads to a
non-concave maximization problem, for which it is very hard to find a result
close to a global maximizer. In order to facilitate finding at least a good
local maximum, we lend from the idea of a continuation of the penalty parameter
in standard topology optimization using the interpretation of the inverse
interpolation model as the RAMP scheme (\ref{eq:ramp}) with $q=(E_0-E_D)/E_D$. 
We start by solving the same problem as
before with $q=(E_0-E_D)/E_D$ and continuously decrease $q$ unto a value of 0.
Note that the lower-level problem fails to be strictly convex for
$q<(E_0-E_D)/E_D$, but that we are still able to obtain locally optimal
solutions using IPOPT. At convergence of the continuation scheme, i.~e. for $q=0$, the interpolation between $E_D$ and $E_0$ becomes linear. 
It can be easily checked that the objective becomes convex in this case. 
As we are maximizing the objective in the lower level problem, this means that have to expect of a worst-case material distribution which is close to binary. 
On the other hand, for $\bb{\delta}\in \{0,1\}^n$, the inverse interpolation model yields
the same effective material properties as the linear interpolation model and thus the same
compliance. This means, the less intermediate values of $\bb{\delta}$ appear
in the worst-case material distribution, when we apply the inverse interpolation model, the closer the gap to the compliance obtained with the linear model is expected to be. Finally, we would like to recall that -- for each fixed topology $\bb{\rho}$ -- the worst-case compliance computed using the inverse model is an upper bound for the
true worst-case compliance (in the sense of a global maximum) using the linear model. Moreover the approximate solution computing using the RAMP continuation scheme constitutes a lower bound for the same value.

We can now apply the RAMP scheme in two different ways. We either can use it to check the conservatism of the result obtained from the solution of the minimax problem based on the inverse interpolation model. In that case, the RAMP scheme is applied only once for the final robustly optimized topology. If the gap between lower and upper bound for the worst case compliance turns out to be small, we can be sure that we did not waste much performance by the effective relaxation of the uncertainty set (cf. again \figref{fig:external-approx} for an illustration). The second possibility is to directly perform the optimization based on the RAMP approach, i.~e. in each iteration we essentially attempt to solve the adversarial problem formulated based on the linear interpolation model. It has to be mentioned that -- from a theoretical point of view -- this is not rigorous as the marginal function might by even discontinuous whenever we fail to find the global maximum of the inner problem. Nevertheless, we follow this second route in the remainder of this section.

Technically, after trying a number of different continuation schemes, we settled on 10
continuation steps from the inverse to the linear model as a reasonable compromise between
a smooth continuation and computational performance. We performed numerical experiments with material parameters $E_0=1$ and $E_D=.01$ as well as total uncertainty bounds varying from $0.0005$ to $0.02$.
It is interesting to note that we did not observe any technical difficulties which could be traced back to a potential discontinuity of the marginal function during numerical computations. The results in terms of relative compliances are summarized in \tabref{tab:rto-ramp}. The associated optimal topologies and worst-case results are illustrated in
\figref{fig:rto-ramp2}. In \tabref{tab:rto-ramp}, for either of the worst-case computations we state three different compliance values: The worst-case compliance obtained using the inverse model (which serves as an upper bound here), the result obtained using the RAMP continuation scheme and the result obtained from attempting to solve the worst-case problem in a brute force way for the linear model. While
we observe that there is only one exceptional case ($D=.001$ for the nominal topology), for which the worst-case compliance obtained by the direct application of the linear model is better than the one obtained by the continuation scheme, in all other cases the continuation scheme turns out to be
by far superior. 
It should be noted that for the direct solution of the
optimization problem with linear interpolation and larger uncertainty bounds IPOPT sometimes
struggled to obtain a solution with values of $\bb{\delta}$ close to 0, explaining the extreme difference in comparison with the continuation scheme. It is also worth to note that for the nominal topology for which a close to binary worst-case material distribution has been already obtained based on the inverse model, the compliance calculated by the continuation scheme comes very close to the upper bound, which indicates that we found a quite good approximation of the true worst case also in the linear interpolation model. Even more surprising is that the situation for the final worst topologies is not much different: also here the gap between lower and upper bounds becomes surprisingly tight. The corresponding compliance values differ for larger choices of $D$ by only 10-20\%. Besides this, the worst-case compliance
values shown in \tabref{tab:rto-ramp} again look rather impressive: in almost all cases the compliance loss can be consistently more than halved for the robust topology compared with
the nominal one, while again the compliance of the worst case topology is only little larger than the nominal topology for the best case material parameters corresponding to $\bb{\delta}\equiv 0$.\\

Looking at \figref{fig:rto-ramp2} we first see again that for an
increasing bound on total material degradation the robust topology becomes thicker
along the support and more and more degraded material appears where the
force is applied. On the other hand, in sharp contrast to the results obtained using the inverse interpolation model, now also for the robustly optimized topologies a binary degradation pattern is observed. As outlined above, this is expected due to the convex nature of the maximization problem using the linear interpolation scheme.


\begin{table*}[] \caption{Compliance increases in comparison to nominal
topology with only stiffest material ($\bb{\delta}\equiv 0$); see \tabref{tab:am-linear} for a detailed explanation of columns 2 to 4. Material parameters are chosen to be $E_0=1$ and $E_D=0.01$. 
Worst-case compliances were computed in three different ways: with a RAMP
continuation scheme in 10 steps from $q=(E_0-E_D)/E_D$ to $q=0$ (``contin'', in bold),
direct solution without continuation for $q=0$ (``direct'') and inverse
interpolation (``inverse'') for comparison.}
\label{tab:rto-ramp}
{
\setlength{\tabcolsep}{1.5em}
\begin{tabular}{|c|c|c|c|c|c|c|c|}


\hline
\makecell{total\\ uncertainty}   & \makecell{w.-c. topo,\\ $\bb{\delta}\equiv 0$} & \multicolumn{3}{c|}{\makecell{nom. topo,\\ worst-case $\bb{\delta}$}} &
\multicolumn{3}{c|}{\makecell{w.-c. topo,\\ worst-case $\bb{\delta}$}}
\\
\hline
  & & contin & direct & inverse & contin & direct & inverse \\
\hline
\hspace{-.6em} $D = 0.0005$\hspace{-.2em} & +0.28\% & \textbf{+32.6\%} & +17.6\%
& +42.3\% & \textbf{+22.4\%} & +15.0\% & +37.7\%
\\
\hline
$D = 0.001$ & +5.75\% & \textbf{+67.0\%} & +70.5\% & +81.6\% & \textbf{+29.6\%}
& +28.7\% & +43.1\%
\\
\hline
$D = 0.002$ & +8.78\% & \textbf{+139\%} & +77.9\% & +158\% & \textbf{+59.0\%} &
+52.4\% & +73.2\%\\
\hline
$D = 0.004$ & +8.62\% & \textbf{+286\%} & +229\% & +306\% & \textbf{+119\%} &
+61.2\% & +135\%\\
\hline
$D = 0.006$ & +7.30\% & \textbf{+425\%} & +281\% & +444\% & \textbf{+180\%} &
+138\% & +207\% \\
\hline
$D = 0.008$ & +7.55\% & \textbf{+554\%} & +380\% & +575\% & \textbf{+239\%} &
+187\% & +270\% \\
\hline
$D = 0.01$ & +7.19\% & \textbf{+675\%} & +440\% & +699\% & \textbf{+296\%} &
+223\% & +330\% \\
\hline
$D = 0.02$ & +12.5\% & \textbf{+1203\%} & +442\% & +1234\% & \textbf{+539\%} &
+449\% & +597\% \\
\hline
\end{tabular}
}
\end{table*}
\begin{figure*}[]
\centering
\subfigure[worst-case distribution of Young's modulus for $D=0.0005$ (nominal topology)]
{\includegraphics[width=.42\textwidth]{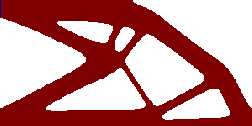}\label{ramp/0.0005/Esc}}\hfill
\raisebox{-9mm}{\subfigure{ \begin{tikzpicture}[scale=1.3]
  \node at (0,0)
  {\includegraphics[width=.025\textwidth]{colormap-E.png}};
  \node[text width=.025\textwidth] at (0.08,1.55) {$1$};
  \node[text width=.025\textwidth] at (0.08,-1.55) {$0$};
  \node[text width=.025\textwidth] at (-0.16,-1.85) {$(E^{\text{eff}}$)};
\end{tikzpicture}
}}\hfill
 \subfigure[worst-case distribution of Young's modulus for $D=0.0005$ (robust topology)]
 {\includegraphics[width=.42\textwidth]{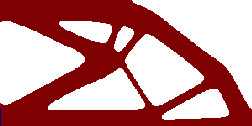}\label{ramp/0.0005/Efc}}
 \subfigure[worst-case distribution of Young's modulus for $D=0.008$ (nominal topology)]
 {\includegraphics[width=.42\textwidth]{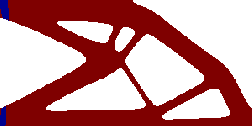}\label{ramp/0.0080/Esc}}\hfill
\subfigure[worst-case distribution of Young's modulus for $D=0.008$ (robust topology)]
 {\includegraphics[width=.42\textwidth]{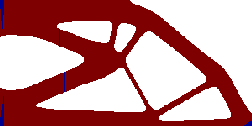}\label{ramp/0.0080/Efc}}
 \subfigure[worst-case distribution of Young's modulus for $D=0.02$ (nominal topology)]
 {\includegraphics[width=.42\textwidth]{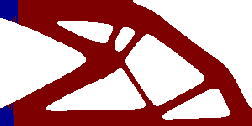}\label{ramp/0.0200/Esc}}\hfill
\subfigure[worst-case distribution of Young's modulus for $D=0.02$ (robust topology)]
 {\includegraphics[width=.42\textwidth]{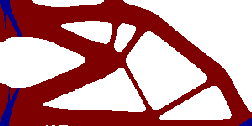}\label{ramp/0.0200/Efc}}
 \caption{Optimization results obtained using the RAMP-type continuation to a linear
 interpolation model. Displayed are worst-case distributions of Young's moduli for different bounds on total material variation. Material parameters are $E_0=1$, $E_D=0.01$. It is clearly observed that enabling the use of the linear interpolation law by the continuation scheme, a binary degeneration pattern is now also obtained for the robustly optimized topologies.}\label{fig:rto-ramp2}
\end{figure*}

\subsection{Computational efficiency}
All computations were done on a machine with an Intel Xeon CPU E3-1280 v2 with
four cores with 3.60 GHz and 32 GB of memory. Despite the very strict optimality
and constraint tolerance of $10^{-10}$ used in IPOPT, the second order algorithm
and the reuse of the previous iterates brought down the iteration numbers for
each solution of the lower-level problem to a reasonable range. An individual lower level iteration took about one second for
$300\times150$ finite elements. For the additive manufacturing example with the
average-based uncertainty set, the first solution of the lower-level problem
took 21-22 interior point iterations, each solution after that only 4-5 and occasionally 6
interior point iterations. For the material degradation examples
with $E_D=0.01$, the first outer iteration needed between 13 and 59 lower-level
iterations, which decreased to no more than 10 iterations after the first 3-8
outer iteration steps. For the RAMP
continuation, IPOPT struggled to solve the lower-level problem as efficiently.
Without a noticeable change between outer iterations, the number for the RAMP continuation increased with growing total uncertainty bound from about 300 iterations ($D=0.0005$) to 1150 ($D=0.02$). It should be mentioned however that the warm start strategy was much less successful in this case, leaving room for improvements in the future.\\
For larger problem sizes, the iteration numbers and computation times are
expected to scale quite well, as the main complexity comes from the solution of
a strictly convex problem with a second-order method. We tested this for the
examples with $E_D=0.01$ on a grid of $600\times300$ finite elements and found a
quite consistent scaling for different total uncertainty bounds. With 4
times the finite element number, average times per outer iteration rose by a
factor of 4.8, while the average number of iterations for the lower-level
problem solution actually decreased by a factor of 0.82.
\section{Conclusion}
We proposed a model and an algorithm for the robust optimal compliance design with unpredictable material variations and applied it
successfully to a number of numerical examples. For additive manufacturing examples the effect of the uncertain material parameters turned out to be minor even in the worst case. This is an interesting result by itself. 
On the other hand, the robust topology optimization accounting for material degradation showed very promising results. A rigorous worst-case degeneration formulation could be
solved with a still manageable increase in computational cost compared to the nominal solution. Moreover, the robustly optimized topologies showed a much better worst-case performance than the nominal ones. 

One partly open point is however that the exact nature of robustness of the optimal result and its performance is not entirely clear. While the worst-case performance for the robust topologies improved strongly, it consistently remained far away from that
of the nominal solution. This might be partly due to the nature of the studied examples, however another reason might be the fact that the convexification we applied for the lower level problem leads to conservative solutions. To get a feeling about the latter, a RAMP continuation strategy
was suggested, which lead to a binary worst-case material degeneration. By this, the gap could be tightened. On the other hand there is no guarantee, that the worst case was identified in this case. Moreover, the continuation strategy came at a substantially increased computational demand. 

Future research directions include studying more elaborate schemes for
the RAMP continuation and possibly different solvers for the lower-level problem
in the non-convex RAMP case. 
For additive
manufacturing, different cost functionals might be considered which are more
sensitive to relatively small changes in the material parameters. One difficulty in this case would be however that even with the inverse interpolation model the inner maximization problem might fail to be convex. As a consequence alternative methods from global optimization would be required to identify the worst-case material distribution.

\begin{acknowledgements}
The authors thank the German Research Foundation (DFG) for funding
this research work within Collaborative Research Centre 814, subproject C2.
\end{acknowledgements}

\newpage

\bibliographystyle{spbasic}
\bibliography{optimization,references}

\begin{thebibliography}{21}
\providecommand{\natexlab}[1]{#1}
\providecommand{\url}[1]{{#1}}
\providecommand{\urlprefix}{URL }
\expandafter\ifx\csname urlstyle\endcsname\relax
  \providecommand{\doi}[1]{DOI~\discretionary{}{}{}#1}\else
  \providecommand{\doi}{DOI~\discretionary{}{}{}\begingroup
  \urlstyle{rm}\Url}\fi
\providecommand{\eprint}[2][]{\url{#2}}

\bibitem[{vdi(2015)}]{vdihandbuch}
 (2015) VDI 3405 Part 2.1: Additive manufacturing processes, rapid
  manufacturing - Laser beam melting of metallic parts: Material data sheet
  aluminium alloy AlSi10Mg. VDI-Gesellschaft Produktion und Logistik, Berlin,
  vDI-Handbuch

\bibitem[{Achtziger and Bends{\o}e(1995)}]{achtziger1995design}
Achtziger W, Bends{\o}e MP (1995) Design for maximal flexibility as a simple
  computational model of damage. Structural optimization 10(3-4):258--268

\bibitem[{Achtziger and Bends{\o}e(1999)}]{achtziger1999optimal}
Achtziger W, Bends{\o}e MP (1999) Optimal topology design of discrete
  structures resisting degradation effects. Structural optimization
  17(1):74--78

\bibitem[{Achtziger et~al(1998)Achtziger, Bends{\o}e, and
  Taylor}]{achtziger1998bounds}
Achtziger W, Bends{\o}e MP, Taylor JE (1998) Bounds on the effect of
  progressive structural degradation. Journal of the Mechanics and Physics of
  Solids 46(6):1055--1087

\bibitem[{Bends{\o}e and D{\'\i}az(1998)}]{bendsoe1998method}
Bends{\o}e MP, D{\'\i}az AR (1998) A method for treating damage related
  criteria in optimal topology design of continuum structures. Structural
  optimization 16(2-3):108--115

\bibitem[{Calafiore and Dabbene(2008)}]{calafiore2008optimization}
Calafiore GC, Dabbene F (2008) Optimization under uncertainty with applications
  to design of truss structures. Structural and Multidisciplinary Optimization
  35(3):189--200

\bibitem[{Chen et~al(2010)Chen, Chen, and Lee}]{chen2010level}
Chen S, Chen W, Lee S (2010) Level set based robust shape and topology
  optimization under random field uncertainties. Structural and
  Multidisciplinary Optimization 41(4):507--524

\bibitem[{Clarke(1975)}]{clarke1975generalized}
Clarke FH (1975) Generalized gradients and applications. Transactions of the
  American Mathematical Society 205:247--262

\bibitem[{Francfort and Marigo(1993)}]{francfort1993stable}
Francfort GA, Marigo JJ (1993) Stable damage evolution in a brittle continuous
  medium. European Journal of mechanics series a solids 12:149--149

\bibitem[{Gauvin and Dubeau(1982)}]{gauvin1982differential}
Gauvin J, Dubeau F (1982) Differential properties of the marginal function in
  mathematical programming. In: Optimality and Stability in Mathematical
  Programming, Springer, pp 101--119

\bibitem[{Hermann~Becker et~al(2016)Hermann~Becker, Hermann~Becker, Dimitrov,
  and Dimitrov}]{hermann2016achievable}
Hermann~Becker T, Hermann~Becker T, Dimitrov D, Dimitrov D (2016) The
  achievable mechanical properties of slm produced maraging steel 300
  components. Rapid Prototyping Journal 22(3):487--494

\bibitem[{HSL(2013)}]{hsl2013collection}
HSL (2013) A collection of fortran codes for large-scale scientific
  computation. See http://www hsl rl ac uk

\bibitem[{Jansen et~al(2014)Jansen, Lombaert, Schevenels, and
  Sigmund}]{jansen2014topology}
Jansen M, Lombaert G, Schevenels M, Sigmund O (2014) Topology optimization of
  fail-safe structures using a simplified local damage model. Structural and
  Multidisciplinary Optimization 49(4):657--666

\bibitem[{Kachanov(1958)}]{kachanov1958time}
Kachanov L (1958) Time of the rupture process under creep conditions. Izv Akad
  Nauk SSSRR, Otd Tech Nauk 8:26--31

\bibitem[{Lazarov et~al(2012)Lazarov, Schevenels, and
  Sigmund}]{lazarov2012topology}
Lazarov BS, Schevenels M, Sigmund O (2012) Topology optimization considering
  material and geometric uncertainties using stochastic collocation methods.
  Structural and Multidisciplinary optimization 46(4):597--612

\bibitem[{Lemaitre(1978)}]{lemaitre1978theorie}
Lemaitre J (1978) Th{\'e}orie m{\'e}canique de l'endommagement isotrope
  appliqu{\'e} {\`a} la fatigue des m{\'e}taux. S{\'e}minaire Mat{\'e}riaux et
  Structures sous chargement cyclique, Ecole Polytechnique, Palaiseau

\bibitem[{Stolpe and Svanberg(2001)}]{stolpe2001alternative}
Stolpe M, Svanberg K (2001) An alternative interpolation scheme for minimum
  compliance topology optimization. Structural and Multidisciplinary
  Optimization 22(2):116--124

\bibitem[{Svanberg(2002)}]{Svanberg:02}
Svanberg K (2002) A class of globally convergent optimization methods based on
  conservative convex separable approximations. SIAM J Opt 12(2):555--573

\bibitem[{W{\"a}chter and Biegler(2006)}]{wachter2006implementation}
W{\"a}chter A, Biegler LT (2006) On the implementation of an interior-point
  filter line-search algorithm for large-scale nonlinear programming.
  Mathematical programming 106(1):25--57

\bibitem[{Wang et~al(2011)Wang, Lazarov, and Sigmund}]{Wang:11:Robust}
Wang F, Lazarov BS, Sigmund O (2011) On projection methods, convergence and
  robust formulations in topology optimization. Struct Multidiscip Opt
  43(6):767--784

\bibitem[{Wegner and Witt(2012)}]{wegner2012betrachtung}
Wegner A, Witt G (2012) Betrachtung zur pulvernutzungsdauer beim laser-sintern
  und einfluss der prozessf{\"u}hrung auf die entstehung von ausschussbauteilen
  (ger). In: RTejournal-Forum f{\"u}r Rapid Technologie, vol 2012

\end{thebibliography}
\end{document}